\documentclass[12pt,leqno]{article}
\usepackage{amsfonts}
\usepackage{amsmath, dsfont, amsthm, amsfonts, amssymb, color}
\usepackage{mathrsfs}
\usepackage{color}
\usepackage{stmaryrd}
\newtheorem{thm}{Theorem}[section]

\newtheorem{lem}[thm]{Lemma}

\newtheorem{rem}[thm]{Remark}
\theoremstyle{definition}

\newcommand{\scr}[1]{\mathscr #1}
\definecolor{wco}{rgb}{0.5,0.2,0.3}

\numberwithin{equation}{section} \theoremstyle{remark}

\newcommand{\ua}{\uparrow}

\title{{\bf
Long Time $\W_0$-$\widetilde{\W}_1$ type Propagation of Chaos for Mean Field Interacting Particle System}\footnote{Supported in
 part by  National Key R\&D Program of China (No. 2022YFA1006000) and NNSFC (12271398,12101390).} }
\author{
{\bf  Xing Huang $^{a)}$,  Fen-Fen Yang $^{b)}$, Chenggui Yuan $^{c)}$  }\\
\footnotesize{ a)Center for Applied Mathematics, Tianjin
University, Tianjin 300072, China}\\
\footnotesize{  xinghuang@tju.edu.cn}\\
 \footnotesize{ b)Department of Mathematics, Shanghai University, Shanghai 200444, China}\\
\footnotesize{ yangfenfen@shu.edu.cn }\\
 \footnotesize{ c)Department of Mathematics, Swansea University, Bay campus, SA1 8EN, UK}\\
\footnotesize{ c.yuan@swansea.ac.uk }}
\begin{document}
\allowdisplaybreaks
\def\R{\mathbb R}  \def\ff{\frac} \def\ss{\sqrt} \def\B{\mathbf
B} \def\W{\mathbb W}
\def\N{\mathbb N} \def\kk{\kappa} \def\m{{\bf m}}
\def\ee{\varepsilon}\def\ddd{D^*}
\def\dd{\delta} \def\DD{\Delta} \def\vv{\varepsilon} \def\rr{\rho}
\def\<{\langle} \def\>{\rangle} \def\GG{\Gamma} \def\gg{\gamma}
  \def\nn{\nabla} \def\pp{\partial} \def\E{\mathbb E}
\def\d{\text{\rm{d}}} \def\bb{\beta} \def\aa{\alpha} \def\D{\scr D}
  \def\si{\sigma} \def\ess{\text{\rm{ess}}}
\def\beg{\begin} \def\beq{\begin{equation}}  \def\F{\scr F}
\def\Ric{\text{\rm{Ric}}} \def\Hess{\text{\rm{Hess}}}
\def\e{\text{\rm{e}}} \def\ua{\underline a} \def\OO{\Omega}  \def\oo{\omega}
 \def\tt{\tilde} \def\Ric{\text{\rm{Ric}}}
\def\cut{\text{\rm{cut}}} \def\P{\mathbb P} \def\ifn{I_n(f^{\bigotimes n})}
\def\C{\scr C}      \def\aaa{\mathbf{r}}     \def\r{r}
\def\gap{\text{\rm{gap}}} \def\prr{\pi_{{\bf m},\varrho}}  \def\r{\mathbf r}
\def\Z{\mathbb Z} \def\vrr{\varrho}
\def\L{\scr L}\def\Tt{\tt} \def\TT{\tt}\def\II{\mathbb I}
\def\i{{\rm in}}\def\Sect{{\rm Sect}}  \def\H{\mathbb H}
\def\M{\scr M}\def\Q{\mathbb Q} \def\texto{\text{o}}
\def\Rank{{\rm Rank}} \def\B{\scr B} \def\i{{\rm i}} \def\HR{\hat{\R}^d}
\def\to{\rightarrow}\def\l{\ell}\def\iint{\int}
\def\EE{\scr E}\def\Cut{{\rm Cut}}
\def\A{\scr A} \def\Lip{{\rm Lip}}
\def\BB{\scr B}\def\Ent{{\rm Ent}}\def\L{\scr L}
\def\R{\mathbb R}  \def\ff{\frac} \def\ss{\sqrt} \def\B{\mathbf
B}
\def\N{\mathbb N} \def\kk{\kappa} \def\m{{\bf m}}
\def\dd{\delta} \def\DD{\Delta} \def\vv{\varepsilon} \def\rr{\rho}
\def\<{\langle} \def\>{\rangle} \def\GG{\Gamma} \def\gg{\gamma}
  \def\nn{\nabla} \def\pp{\partial} \def\E{\mathbb E}
\def\d{\text{\rm{d}}} \def\bb{\beta} \def\aa{\alpha} \def\D{\scr D}
  \def\si{\sigma} \def\ess{\text{\rm{ess}}}
\def\beg{\begin} \def\beq{\begin{equation}}  \def\F{\scr F}
\def\Ric{\text{\rm{Ric}}} \def\Hess{\text{\rm{Hess}}}
\def\e{\text{\rm{e}}} \def\ua{\underline a} \def\OO{\Omega}  \def\oo{\omega}
 \def\tt{\tilde} \def\Ric{\text{\rm{Ric}}}
\def\cut{\text{\rm{cut}}} \def\P{\mathbb P} \def\ifn{I_n(f^{\bigotimes n})}
\def\C{\scr C}      \def\aaa{\mathbf{r}}     \def\r{r}
\def\gap{\text{\rm{gap}}} \def\prr{\pi_{{\bf m},\varrho}}  \def\r{\mathbf r}
\def\Z{\mathbb Z} \def\vrr{\varrho}
\def\L{\scr L}\def\Tt{\tt} \def\TT{\tt}\def\II{\mathbb I}
\def\i{{\rm in}}\def\Sect{{\rm Sect}}  \def\H{\mathbb H}
\def\M{\scr M}\def\Q{\mathbb Q} \def\texto{\text{o}} \def\LL{\Lambda}
\def\Rank{{\rm Rank}} \def\B{\scr B} \def\i{{\rm i}} \def\HR{\hat{\R}^d}
\def\to{\rightarrow}\def\l{\ell}
\def\8{\infty}\def\I{1}\def\U{\scr U} \def\n{{\mathbf n}}
\maketitle

\begin{abstract} In this paper, a general result on the long time $\W_0$-$\widetilde{\W}_1$ type propagation of chaos, propagation of chaos with regularization effect, for mean field interacting particle system driven by L\'{e}vy noise is derived, where $\W_0$ is one half of the total variation distance while $\widetilde{\W}_1$ is the $L^1$-Wasserstein distance.  By using the method of coupling, the general result is applied to mean field interacting particle system driven by multiplicative Brownian motion and additive $\alpha(\alpha>1)$-stable noise respectively, where the non-interacting drift is assumed to be dissipative in long distance and the initial distribution of interacting particle system converges to that of the limit equation in $\widetilde{\W}_1$.
 \end{abstract}

 \vskip 2cm
\section{Introduction}

Let $(E,\rho)$ be a Polish space and $o$ be a fixed point in $E$. Let $\scr P(E)$ be the set of all probability measures on $E$ equipped with the weak topology. For $p>0$, let
$$\scr P_p(E):=\big\{\mu\in \scr P(E): \mu(\rho(o,\cdot)^p)<\infty\big\},$$
and define the $L^p$-Wasserstein distance
$$\W_p(\mu,\nu)= \inf_{\pi\in \mathbf{C}(\mu,\nu)} \bigg(\int_{E\times E} \rho(x,y)^p \pi(\d x,\d y)\bigg)^{\ff 1 {p\vee 1}},\ \  \mu,\nu\in \scr P_p(E), $$ where $\mathbf{C}(\mu,\nu)$ is the set of all couplings of $\mu$ and $\nu$. When $p>0$, $(\scr P_p(E), \W_p)$ is a polish space. To discuss the interacting particle system, we will consider the product space $E^k$ with $k\geq 1$.
For any $k\geq 1$, define
$$\rho_1(x,y)=\sum_{i=1}^k\rho(x^i,y^i),\ \ x=(x^1,x^2,\cdots,x^k), y=(y^1,y^2,\cdots,y^k)\in E^k$$
and define
$$\widetilde{\W}_1(\mu,\nu)= \inf_{\pi\in \mathbf{C}(\mu,\nu)} \bigg(\int_{E^k\times E^k} \rho_1(x,y) \pi(\d x,\d y)\bigg),\ \  \mu,\nu\in \scr P_1(E^k), $$
where $\scr P_1(E^k)=\big\{\mu\in \scr P(E^k): \mu(\rho_1(\mathbf{o},\cdot))<\infty\big\}$ for $\mathbf{o}=(o,o,\cdots,o)\in E^k$.

We will also use the total variation distance:
$$\|\gamma-\tilde{\gamma}\|_{var}=\sup_{\|f\|_\infty\leq 1}|\gamma(f)-\tilde{\gamma}(f)|,\ \ \gamma,\tilde{\gamma}\in \scr P(E).$$
In particular, for any $\gamma,\tilde{\gamma}\in \scr P(\R^n)$, since $C^2_b(\R^n)$, the functions from $\R^n$ to $\R$ having bounded and continuous up to second order derivatives , is dense in $\scr B_b(\R^n)$ under  $L^1(\gamma+\tilde{\gamma})$, we have
\begin{align}\label{tvc}\|\gamma-\tilde{\gamma}\|_{var}=\sup_{\|f\|_\infty\leq 1, f\in C_b^2(\R^n)}|\gamma(f)-\tilde{\gamma}(f)|,\ \ \gamma,\tilde{\gamma}\in \scr P(\R^n).
\end{align}
In addition, it holds
$$\|\gamma-\tilde{\gamma}\|_{var}=2\W_0(\tilde{\gamma},\gamma):=2\inf_{\pi\in \mathbf{C}(\gamma,\tilde{\gamma})}\int_{\R^n\times \R^n} 1_{\{x\neq y\}} \pi(\d x,\d y),\ \ \gamma,\tilde{\gamma}\in \scr P(\R^n).$$

Let $Z_t$ be an $n$-dimensional L\'{e}vy process on some complete filtration probability space $(\Omega, \scr F, (\scr F_t)_{t\geq 0},\P)$. Recall that for a general $n$-dimensional L\'{e}vy process, its characteristic function has the form
$$\E\e^{i\<\xi,Z_t\>}=\exp\left\{i\<\eta,\xi\>t-\frac{1}{2}\<a\xi,\xi\>t+t\int_{\R^n-\{0\}}\e^{i\<z,\xi\>}-1-i\<z,\xi\> 1_{\{|z|\leq 1\}}\nu(\d z)\right\},\ \ \xi\in\R^n,$$
where $\eta\in\R^n$, $a$ is an $n\times n$ non-negative definite symmetric matrix and $\nu$ is the L\'{e}vy measure satisfying
$$\int_{\R^n}(1\wedge |z|^2)\nu(\d z)<\infty.$$
Let $b:[0,\infty)\times \R^d\times\scr P(\R^d)\to\R^d$, $\sigma:[0,\infty)\times \R^d\to\R^d\otimes\R^{n}$ are measurable and are bounded on bounded set.
Let $N\ge1$ be an integer and $(Z^i_t)_{1\le i\le N}$ be i.i.d.\,copies of $Z_t.$ Consider the non-interacting particle system:
\begin{align}\label{GPS}\d X_t^i= b_t(X_t^i, \L_{X_t^i})\d t+  \sigma_t(X^i_{t-}) \d Z^i_t,\ \ 1\leq i\leq N,
\end{align}
and the mean field interacting particle system
\begin{align}\label{GPS00}\d X^{i,N}_t=b_t(X_t^{i,N}, \hat\mu_t^N)\d t+\sigma_t(X^{i,N}_{t-}) \d Z^i_t,\ \ 1\leq i\leq N,
\end{align}
where $\L_{X_t^i}$ is the distribution of $X_t^i$ while $\hat\mu_t^N$ stands for the empirical distribution of $(X_t^{i,N})_{1\leq i\leq N}$, i.e.
\begin{equation*}
 \hat\mu_t^N =\ff{1}{N}\sum_{j=1}^N\dd_{X_t^{j,N}}.
 \end{equation*}
 Note that \eqref{GPS} consists of $N$ independent McKean-Vlasov SDEs, which are written as
 \begin{align}\label{GPSMV}\d X_t= b_t(X_t, \L_{X_t})\d t+\sigma_t(X_{t-})\d Z_t.
\end{align}
\eqref{GPSMV} was first introduced in \cite{McKean}. Kac's chaotic property, also called Boltzmann's property, was introduced in \cite{Kac} to derive the homogeneous Boltzmann equation by taking the limit of the master equation of Poisson-like process.
When \eqref{GPSMV} and \eqref{GPS00} are well-posed, for any $\mu\in\scr P(\R^d)$, let $P_t^\ast\mu$ be the distribution of the solution to \eqref{GPSMV} with initial distribution $\mu$, and for any exchangeable $\mu^N\in\scr P((\R^d)^N)$, $1\leq k\leq N$, $(P_t^{[k],N})^\ast\mu^N$
be the distribution of $(X_t^{i,N})_{1\leq i\leq k}$ with initial distribution $\mu^N$. Moreover, let $\mu^{\otimes k}$ denote the $k$ independent product of $\mu$, i.e. $\mu^{\otimes k}=\prod_{i=1}^k\mu$. For any $1\leq k\leq N$, let $\pi_k$ be the projection from $(\R^d)^{N}$ to $(\R^d)^{k}$ defined by
$$\pi_k(x)=(x^1,x^2,\cdots,x^k),\ \ x=(x^1,x^2,\cdots,x^N)\in(\R^d)^{N}.$$
Then it is not difficult to see that
\begin{align}\label{margi}(P_t^{[k],N})^\ast\mu^N=\{(P_t^{[N],N})^\ast\mu^N\}\circ(\pi_k)^{-1},\ \ 1\leq k\leq N.
\end{align}
Throughout the paper, we assume that the initial distribution of \eqref{GPS00} is exchangeable.

Let us recall some progress on the propagation of chaos. There are fruitful results in the case $Z_t^i=W_t^i$, $n$-dimensional Brownian motion. When $b_t(x,\mu)=\int_{\R^d}\tilde{b}_t(x-y)\mu(\d y)$ for some function $\tilde{b}$ being Lipschitz continuous in spatial variable uniformly in time variable, \cite{SZ} adopts the synchronous coupling method to explore the quantitative propagation of chaos in strong convergence. When $n=d$, $\sigma=I_{d\times d}$,  the entropy method was introduced in \cite{BJW,JW,JW1} to derive the quantitative entropy-entropy propagation of chaos:
\begin{align}\label{e-e}\mathrm{Ent}((P_t^{[k],N})^\ast\mu^N_0| (P_t^\ast\mu_0)^{\otimes k})\leq \frac{k}{N}\mathrm{Ent}(\mu^N_0|\mu_0^{\otimes N})+\frac{ck}{N},\ \ t\in[0,T],1\leq k\leq N,
\end{align}
for some constant $c>0$ depending on $T>0$, here the relative entropy of two probability measures is defined as
$$\mathrm{Ent}(\nu|\mu)=\left\{
  \begin{array}{ll}
    \nu(\log(\frac{\d \nu}{\d \mu})), & \hbox{$\nu\ll\mu$;} \\
    \infty, & \hbox{otherwise.}
  \end{array}
\right.$$
The idea of the entropy method is to derive the evolution on $t$ of $\mathrm{Ent}((P_t^{[N],N})^\ast\mu^N_0| (P_t^\ast\mu_0)^{\otimes N})$ from the Fokker-Planck-Kolmogorov equations for $(P_t^{[N],N})^\ast\mu^N_0$ and $(P_t^\ast\mu_0)^{\otimes N}$ respectively. The procedure relies on the chain rule of the Laplacian operator. Then \eqref{e-e} is obtained by the tensor property of relative entropy:
$$\mathrm{Ent}((P_t^{[k],N})^\ast\mu^N_0| (P_t^\ast\mu_0)^{\otimes k})\leq \frac{k}{N}\mathrm{Ent}((P_t^{[N],N})^\ast\mu^N_0| (P_t^\ast\mu_0)^{\otimes N}),\ \ t\in[0,T], 1\leq k\leq N.$$
Recently, \cite{LL} applies the BBGKY argument to estimate $\mathrm{Ent}((P_t^{[k],N})^\ast\mu^N_0| (P_t^\ast\mu_0)^{\otimes k})$ directly and then derives the sharp rate $\frac{k^2}{N^2}$ instead of $\frac{k}{N}$ for entropy-entropy propagation of chaos in the case of Lipchitzian or bounded interaction. Basing on \cite{LL}, combining the BBGKY argument and the uniform in time log-Sobolev inequality for $\L_{X_t^i}$, \cite{L21} shows the sharp long time entropy-entropy propagation of chaos, which together with the Pinsker inequality
$$\|\mu-\nu\|_{var}^2\leq 2\mathrm{Ent}(\nu|\mu)$$
implies the sharp long time $\W_0$-entropy type propagation of chaos, i.e.
\begin{align}\label{e-e1}\|(P_t^{[k],N})^\ast\mu^N_0- (P_t^\ast\mu_0)^{\otimes k}\|_{var}\leq \sqrt{2 \mathrm{Ent}(\mu_0^N\circ(\pi_k)^{-1}|\mu_0^{\otimes k})}+\frac{ck}{N},\ \ t\geq 0,1\leq k\leq N.
\end{align}

Quite recently, combining Wang's Harnack inequality with power, \cite{MRW} presents explicit conditions for the uniform in time log-Sobolev inequality for $\L_{X_t^i}$ in the noncovex frame, which together with \cite{LL} implies the long time entropy-entropy propagation of chaos with sharp rate $\frac{k^2}{N^2}$.
For kinetic mean field interacting particle system, the authors in \cite[Theorem 2.3]{CLRW} adopt the synchronous coupling technique to derive $\W_2$-$\W_2$ type propagation of chaos, and then apply the log-Harnack inequality (which is equivalent to the entropy-cost estimate) due to the coupling by change of measure to obtain the entropy-$\W_2^2$ type propagation of chaos, i.e. for $s+1\geq t>s\geq 0$,
$$\mathrm{Ent}((P_t^{[k],N})^\ast\mu^N_0| (P_t^\ast\mu_0)^{\otimes k})\leq \frac{C_1k}{N(t-s)^{3}}\W_2((P_s^{[N],N})^\ast\mu^N_0, (P_s^\ast\mu_0)^{\otimes N})^2+\frac{k}{N}C_2,\ \ 1\leq k\leq N $$
for some constant $C_2$ depending on $t,s$ and the variance of $P_s^\ast\mu_0$,
which together with the Pinsker's inequality implies the $\W_0$-$\W_2$ type propagation of chaos, i.e. for $s+1\geq t>s\geq 0$,
$$\W_0((P_t^{[k],N})^\ast\mu^N_0,(P_t^\ast\mu_0)^{\otimes k})\leq \frac{\sqrt{C_1k}}{\sqrt{N}(t-s)^{\frac{3}{2}}}\W_2((P_s^{[N],N})^\ast\mu^N_0, (P_s^\ast\mu_0)^{\otimes N})+\sqrt{\frac{k}{N}C_2}, \ \ 1\leq k\leq N.$$
 Both the entropy-$\W_2^2$ type and $\W_0$-$\W_2$ type propagation of chaos reflect the regularization effect of the stochastic noise which can allow the initial distribution $\mu_0$ to be singular with respect to $\mu_0^N$.

It is known that when the drifts are convex (namely uniformly dissipative), the long time $\W_2$-$\W_2$ type propagation of chaos can be derived by the synchronous coupling even in multiplicative noise case. However, as far as we know, the long time $\W_2$-$\W_2$ type propagation of chaos, the entropy-$\W_2^2$ type and $\W_0$-$\W_2$ type propagation of chaos in the non-convex (namely partially dissipative) and multiplicative Brownian motion case are still open, which seem more interesting and challenged.

When there exists a partially dissipative non-interacting drift, instead of the synchronous coupling, the authors in \cite{DEGZ} develop the asymptotic reflection coupling to derive the long time $\widetilde{\W}_1$-$\widetilde{\W}_1$ type propagation of chaos:
\begin{align}\label{L1L1}
&\widetilde{\W}_1((P_t^{[k],N})^\ast\mu^N_0,(P_t^\ast\mu_0)^{\otimes k})\leq c\varepsilon(t)\frac{k}{N}\widetilde{\W}_1(\mu^N_0,\mu_0^{\otimes N})+c\frac{k}{\sqrt{N}},\ \ t\geq 0,1\leq k\leq N
\end{align}
for some constants $c>0$ and $\lim_{t\to\infty}\varepsilon(t)=0$. The asymptotic reflection coupling can date back to \cite{W15}, where it was constructed to study the ergodicity of nonlinear monotone SPDEs. One can also refer to \cite[Theorem 2.11(b)]{LWZ} for propagation of chaos in $\widetilde{\W}_1$ if $\mu_0^N=\mu_0^{\otimes N}$.
The asymptotic reflection coupling is also applied to study the long time behavior of one-dimensional McKean-Vlasov SDEs with common noise in \cite{BW}.

Compared with the above significant progresses on propagation of chaos in Brownian motion noise case, there are fewer results on the propagation of chaos in general L\'{e}vy noise case. \cite{LMW} derives the long time $\widetilde{\W}_1$-$\widetilde{\W}_1$ type propagation of chaos \eqref{L1L1} for interacting particle system driven by L\'{e}vy noise, where the asymptotic refined basic coupling is used. We should point out that in \cite{DEGZ} and \cite[Theorem 1.2]{LMW}, the initial distribution $\mu_0^N$ of $\{X_0^{i,N}\}_{1\leq i\leq N}$ satisfies $\mu_0^N=\tilde{\mu}_0^{\otimes N}$ for some $\tilde{\mu}_0\in\scr P_1(\R^d)$, which together with the fact
$$\widetilde{\W}_1(\tilde{\mu}_0^{\otimes k},\mu_0^{\otimes k})=\frac{k}{N}\widetilde{\W}_1(\tilde{\mu}_0^{\otimes N},\mu_0^{\otimes N})$$
implies that \eqref{L1L1} becomes
\begin{align*}
&\widetilde{\W}_1((P_t^{[k],N})^\ast\tilde{\mu}_0^{\otimes N},(P_t^\ast\mu_0)^{\otimes k})\leq c\varepsilon(t)\widetilde{\W}_1(\tilde{\mu}_0^{\otimes k},\mu_0^{\otimes k})+c\frac{k}{\sqrt{N}},\ \ t\geq 0,1\leq k\leq N.
\end{align*}

However, to our knowledge, the quantitative propagation of chaos in relative entropy in the L\'{e}vy noise even in $\alpha$-stable noise case is still open. The difficulty lies in that the chain rule for $-(-\Delta)^{\frac{\alpha}{2}}$ is not explicit so that the entropy method in \cite{BJW,JW,JW1} seems unavailable in the $\alpha$-stable noise case. To illustrate this point precisely, let $n=d$, $b^i:\R^d\to\R^d,i=1,2$ be measurable and consider
$$\d Y^i_t=b^i(Y_t^i)\d t+\d Z_t.$$
Let $\scr L^Z$ be the generator of $Z_t$. Denote $P_t^i$ the associated semigroup to $Y_t^i$ and
$$\scr L^i=\<b^i,\nabla\>+\scr L^Z,\ \ i=1,2.$$
Let $f\in C_c^\infty(\R^d)$, the set of all smooth functions on $\R^d$ with compact support. Assume that the Kolmogorov forward equation for $P_t^1$ and the Kolmogorov backward equation for $P_t^2$ hold respectively, i.e.
\begin{align}\label{KFB}\frac{\d P_t^1 f}{\d t}=P_t^1 \scr L^1 f,\ \ \frac{\d P_t^2 f}{\d t}=\scr L^2 P_t^2  f.
\end{align}
Let $\Phi\in C^2((0,\infty))$. By formal calculation and \eqref{KFB}, it holds
\begin{align}\label{sem12}\nonumber P_t^1 \Phi(f)-\Phi(P_t^2 f)&=\int_0^t \left(\frac{\d P_s^1 \Phi(P_{t-s}^2f)}{\d s}\right)\d s\\
&=\int_0^t [P_s^1 \{\Phi'(P_{t-s}^2f)\<b^1-b^2,\nabla P_{t-s}^2f\>\}]\d s\\
\nonumber&+\int_0^t [P_s^1\L^Z \{\Phi(P_{t-s}^2f)\}- P_s^1 \{\Phi'(P_{t-s}^2f)(\L^ZP_{t-s}^2f)\}]\d s.
\end{align}
When $\scr L^Z=\Delta$, \eqref{sem12} together with the chain rule
\begin{align}\label{chain}\Delta \{\Phi(f)\}=\Phi''(f)|\nabla f|^2+\Phi'(f)\Delta f
\end{align}
implies that
\begin{align*}P_t^1 \Phi(f)-\Phi(P_t^2 f)
&\leq \int_0^t [P_s^1 \{\Phi'(P_{t-s}^2f)\<b^1-b^2,\nabla P_{t-s}^2f\>\}]\d s\\
&+\int_0^t [P_s^1 \{\Phi''(P_{t-s}^2f)|\nabla P_{t-s}^2f|^2\}]\d s.
\end{align*}
In particular, when $\Phi(x)=\log x$, the Cauchy-Schwartz inequality yields that
\begin{align*}P_t^1 \log f-\log(P_t^2 f)
&\leq \frac{1}{2}\int_0^t  [P_s^1|b^1-b^2|^2]\d s-\frac{1}{2}\int_0^t [P_s^1 \{|P_{t-s}^2f|^{-2}|\nabla P_{t-s}^2f|^2\}]\d s\\
&\leq \frac{1}{2}\int_0^t  [P_s^1|b^1-b^2|^2]\d s.
\end{align*}
This inequality coincides with the estimate derived by the Girsanov transform. One can also refer to \cite{23RW} for the entropy estimates of two diffusion processes with different drifts as well as diffusion coefficients.

Unfortunately, when $\alpha\in(0,2)$, $(-\Delta)^{\frac{\alpha}{2}} \{\Phi(f)\}$ is not so explicit as in \eqref{chain} even for $\Phi(x)=\log x$. Hence, it seems rather hard to derive an entropy estimate of $\alpha$-stable process with two different drifts.
Fortunately, when $\Phi(x)=x$, it follows from \eqref{sem12} that
\begin{align}\label{Dum}P_t^1 f-P_t^2 f=\int_0^t [P_s^1\{\L^1-\L^2\}P_{t-s}^2f]\d s=\int_0^t [P_s^1\{\<b^1-b^2,\nabla P_{t-s}^2f\>\}]\d s.
\end{align}
This may shed light on the possibility to derive the $\W_0$-$\widetilde{\W}_1$ type propagation of chaos in the L\'{e}vy noise case. \eqref{Dum} is called the Duhamel formula in the literature, which can date back to \cite[(3a)]{MS}.



The contribution of this paper is to investigate the long time $\W_0$-$\widetilde{\W}_1$ type propagation of chaos for mean field interacting particle system driven by L\'{e}vy noise, which is weaker than the challenged entropy-$\W_2^2$ type propagations of chaos but also inspiring in this direction. To this end, we will first derive a $\W_0$-$\widetilde{\W}_1$ type propagation of chaos in short time, i.e.
\begin{align*}\|(P_t^{[k],N})^\ast\mu^N_0-(P_t^\ast\mu_0)^{\otimes k}\|_{var}\leq ck\frac{\widetilde{\W}_1(\mu^N_0,\mu_0^{\otimes N})}{N}+ck\Delta(N),\ \ 1\leq k\leq N,t\in[0,T]
\end{align*}
with $\lim_{N\to\infty}\Delta(N)=0$,
and then utilize the long time $\widetilde{\W}_1$-$\widetilde{\W}_1$ propagation of chaos \eqref{L1L1} and the semigroup property
$$(P_{t+s}^{[k],N})^\ast\mu_0^N=(P_{s}^{[k],N})^{\ast}\{(P_{t}^{[N],N})^\ast\mu_0^N\},\ \ P_{t+s}^\ast\mu_0=P_s^\ast P_{t}^\ast\mu_0,\ \ 1\leq k\leq N, s\geq 0,t\geq 0$$
to derive $\W_0$-$\widetilde{\W}_1$ type propagation of chaos in long time.

The paper is organized in the following: In Section 2, we first give a general result on the long time $\W_0$-$\widetilde{\W}_1$ type propagation of chaos, and the theory is then applied in multiplicative Brownian motion case and additive $\alpha$-stable noise case in Section 3 and Section 4 respectively. In Section 5, we will provide some auxiliary lemmas which will be used in the proof of the main results.
\section{A general result on long time $\W_0$-$\widetilde{\W}_1$ type propagation of chaos}
Let $b^{(0)}:\R^d\to\R^d$, $b^{(1)}:\R^d\times\R^d\to\R^d$ and $\sigma:\R^d\to\R^d\otimes\R^{n}$ be measurable and bounded on bounded set. Recall that $(Z^i_t)_{1\le i\le N}$ are i.i.d. $n$-dimensional L\'{e}vy processes. Consider the mean field interacting particle system
\begin{align} \label{al1}
\d X^{i,N}_t=b^{(0)}(X_t^{i,N})\d t+\frac{1}{N}\sum_{m=1}^Nb^{(1)}(X_t^{i,N},X_t^{m,N})\d t+\sigma(X_{t-}^{i,N})\d Z_t^i,\ \ 1\leq i\leq N,
\end{align}
and non-interacting particle system
\begin{equation}\label{al}
 \d X_t^i=b^{(0)}(X_t^i)\d t+\int_{\R^d}b^{(1)}(X_t^i,y)\L_{X_t^i}(\d y)\d t+\sigma(X_{t-}^i)\d Z_t^i,\ \ 1\leq i\leq N.
\end{equation}
We assume that SDEs \eqref{al1} and  \eqref{al} are well-posed. As in Section 1,  let $P_t^\ast\mu_0=\L_{X_t^i}$ with $\L_{X_0^i}=\mu_0\in\scr P(\R^d)$, which is independent of $i$. For simplicity, we denote $\mu_t=P_t^\ast\mu_0$. For any exchangeable $\mu^N\in\scr P((\R^d)^N)$, $1\leq k\leq N$, $(P_t^{[k],N})^\ast\mu^N$
denotes the distribution of $(X_t^{i,N})_{1\leq i\leq k}$ from initial distribution $\mu^N$.

To derive the long time $\W_0$-$\widetilde{\W}_1$ type propagation of chaos,
for any $s\geq 0$, consider the decoupled SDE
\begin{align}\label{deSDE}\d X_{s,t}^{i,\mu,z}&=b^{(0)}(X_{s,t}^{i,\mu,z})\d t+\int_{\R^d}b^{(1)}(X_{s,t}^{i,\mu,z},y)\mu_t(\d y)\d t+\sigma(X_{s,t-}^{i,\mu,z})\d Z_t^i, \ \ t\geq s
\end{align}
with $X_{s,s}^{i,\mu,z}=z\in\R^d$.
Let
$$P_{s,t}^{i,\mu} f(z):=\E f(X_{s,t}^{i,\mu,z}), \ \ f\in \scr B_b(\R^{d}),z\in\R^d,i\geq 1,0\leq s\leq t.$$
We also assume that \eqref{deSDE} is well-posed so that $P_{s,t}^{i,\mu}$ does not depend on $i$ and we denote \begin{align}\label{Pmt}P_{s,t}^{\mu}=P_{s,t}^{i,\mu},\ \ i\geq 1.
 \end{align}Moreover, for any $x=(x^1,x^2,\cdots,x^k)\in (\R^{d})^k, F\in \scr B_b((\R^{d})^k)$, and $(s_1,s_2,\cdots,s_k)\in[0,t]^k$ define
\begin{align}\label{Tenso1}(P_{s_1,t}^\mu\otimes P_{s_2,t}^\mu\otimes\cdots\otimes P_{s_k,t}^\mu) F(x):=\E F(X_{s_1,t}^{1,\mu,x^1},X_{s_2,t}^{2,\mu,x^2},\cdots,X_{s_k,t}^{k,\mu,x^k}).
\end{align}
In particular, we denote
\begin{align}\label{Tenso}(P_{s,t}^\mu)^{\otimes k} F(x):=(P_{s,t}^\mu\otimes P_{s,t}^\mu\otimes\cdots\otimes P_{s,t}^\mu) F(x),\ \ 0\leq s\leq t.
\end{align}
For simplicity, we write $P_{t}^\mu =P_{0,t}^\mu $.
For any $F\in C^1((\R^d)^k)$, $1\leq i\leq k$, $x=(x^1,x^2,\cdots,x^k)\in(\R^d)^k$, let $\nabla_i F(x)$ denote the gradient with respect to $x^i$. We now state the general result.
\begin{thm}\label{POCin} Let $\mu_0^N\in\scr P_{1}((\R^d)^N)$ be exchangeable and $\mu_0\in \scr P_{p}(\R^d)$ for some $p\geq 1$. Assume that the following conditions hold.
\begin{enumerate}
\item[(i)] For any $1\leq k\leq N$, $F\in C_b^2((\R^d)^k)$ with $\|F\|_{\infty}\leq 1$, $t\geq 0$, it holds
\begin{align}\label{DUH0}
\nonumber&\int_{(\R^d)^k}F(x)\{(P_t^{[k],N})^\ast\mu^N_0\}(\d x)-\int_{(\R^d)^k}\{(P_{t}^\mu)^{\otimes k} F\}(x)(\mu^N_0\circ\pi_k^{-1})(\d x)\\
&=\int_0^t\sum_{i=1}^k\int_{(\R^d)^N}\bigg\<B^i_s(x),[\nabla_{i}(P^\mu_{s,t})^{\otimes k}F](\pi_k(x))\bigg\>\{(P_s^{[N],N})^\ast\mu^N_0\}(\d x)\d s
    \end{align}
    with
    $$B^i_s(x)=\frac{1}{N}\sum_{m=1}^Nb^{(1)}(x^i,x^m) -\int_{\R^d}b^{(1)}(x^i,y)\mu_s(\d y),\ \ x=(x^1,x^2,\cdots,x^N)\in(\R^d)^N.$$
\item[(ii)] There exists a measurable function $\varphi:(0,\infty)\to (0,\infty)$ with
$\int_0^T\varphi(s)\d s<\infty, T>0$ such that
\begin{align}\label{gra0t}|\nabla P^\mu_{r,t} f|\leq \varphi((t-r)\wedge1)\|f\|_\infty,\ \ f\in\scr B_b(\R^d), 0\leq r<t.
\end{align}
\item[(iii)] There exist an increasing function $g:(0,\infty)\to (0,\infty)$ and a decreasing $\Delta:(0,\infty)\to (0,\infty)$ with $\lim_{N\to\infty}\Delta(N)=0$ such that
\begin{align}\label{CON}
\nonumber&\int_{(\R^d)^N} \left|B^1_s(x)\right|\{(P_s^{[N],N})^\ast\mu^N_0\}(\d x)\\
&\leq g(s)\left\{\frac{1}{N}\widetilde{\W}_1(\mu^N_0,\mu_0^{\otimes N})+\Delta(N)\{1+\{\mu_0(|\cdot|^{p})\}^{\frac{1}{p}}\}\right\}.
\end{align}
\item[(iv)] There exist functions $\varepsilon:(0,\infty)\to (0,\infty)$ with $\lim_{t\to\infty}\varepsilon(t)=0$ and $\tilde{\Delta}:(0,\infty)\to (0,\infty)$ with $\lim_{N\to\infty}\tilde{\Delta}(N)=0$ such that
\begin{align}\label{CMtty}\nonumber&\widetilde{\W}_1((P_t^{[N],N})^\ast\mu^N_0,(P_t^\ast\mu_0)^{\otimes N})\\
&\leq \varepsilon(t)\widetilde{\W}_1(\mu^N_0,\mu_0^{\otimes N})+\{1+\{\mu_0(|\cdot|^{p})\}^{\frac{1}{p}}\}N\tilde{\Delta}(N),\ \ t\geq 0.
\end{align}
Moreover, there exists a constant $c_0>0$ such that
\begin{align}\label{unf}\sup_{t\geq 0}(P_t^\ast\mu_0)(|\cdot|^{p})<c_0(1+\mu_0(|\cdot|^p)).
\end{align}
\end{enumerate}
 Then there exists a constant $c>0$ independent of $t$ and $N$ such that
\begin{align}\label{CMY}\nonumber&\|(P_t^{[k],N})^\ast\mu^N_0-(P_t^\ast\mu_0)^{\otimes k}\|_{var}\\
&\leq ck\varepsilon(t-1)\frac{\widetilde{\W}_1(\mu^N_0,\mu_0^{\otimes N})}{N}\\
\nonumber&+c\{1+\{\mu_0(|\cdot|^{p})\}^{\frac{1}{p}}\}k(\Delta(N)+\tilde{\Delta}(N)),\ \ 1\leq k\leq N,t\geq 1.
\end{align}
\end{thm}
\begin{rem} (1) Different from the classical Duhamel formula \eqref{Dum}, the left hand side of \eqref{DUH0} only contains $(P_t^{[k],N})^\ast$ while the right hand side involves in $(P_t^{[N],N})^\ast$ due to the interaction. In Section 3 and Section 4 below, explicit conditions will be presented on the coefficients to ensure \eqref{DUH0}.

(2) In the study of entropy-entropy propagation of chaos in $\cite{BJW,JW,JW1}$ or $\widetilde{\W}_1$-$\widetilde{\W}_1$ type propagation of chaos \eqref{L1L1}, due to the tensor property of relative entropy or the property
\begin{align}\label{teo}\widetilde{\W}_1((P_t^{[k],N})^\ast\mu^N_0,(P_t^\ast\mu_0)^{\otimes k})\leq \frac{k}{N}\widetilde{\W}_1((P_t^{[N],N})^\ast\mu^N_0,(P_t^\ast\mu_0)^{\otimes N}),
\end{align}
one can derive the estimate of $\mathrm{Ent}((P_t^{[N],N})^\ast\mu^N_0| (P_t^\ast\mu_0)^{\otimes N})$ or $\widetilde{\W}_1((P_t^{[N],N})^\ast\mu^N_0,(P_t^\ast\mu_0)^{\otimes N})$ for $N$ particles first and then obtain the local propagation of chaos for $k$ particles with $k\leq N$.
However, \eqref{teo} does not hold if $\widetilde{\W}_1$ is replaced $\W_0$. This is the reason why we directly consider $k$ particles with $k\leq N$ instead of $N$ particles in \eqref{DUH0}.

(3) As a hot topic related to the long time propagation of chaos, the ergodicity for McKean-Vlasov SDEs attracts much attention, see for instance \cite{LMW,Song,W23} and references therein for more details.
\end{rem}
\begin{proof} Let $F\in C_b^2((\R^d)^k)$ with $\|F\|_\infty\leq 1$.
For any $(x^1,x^2,\cdots,x^{i-1},x^{i+1},\cdots,x^k)\in(\R^{d})^{k-1}$, define
\begin{align*}&[\scr T_{s,t}^{x^1,x^2,\cdots,x^{i-1},x^{i+1},\cdots,x^k}] F(z)\\
&=\E F(X_{s,t}^{1,\mu,x^1},X_{s,t}^{2,\mu,x^2},\cdots,X_{s,t}^{i-1,\mu,x^{i-1}}, z,X_{s,t}^{i+1,\mu,x^{i+1}}\cdots,X_{s,t}^{k,\mu,x^k}),\ \ z\in\R^d,\ \ 0\leq s\leq t.
\end{align*}
This together with \eqref{Tenso} and \eqref{gra0t} implies that
\begin{align}\label{kb2}\nonumber|\nabla_{i}(P^\mu_{s,t})^{\otimes k}F|(x^1,x^2,\cdots,x^k)&=|\nabla \{P_{s,t}^\mu[\scr T_{s,t}^{x^1,x^2,\cdots,x^{i-1},x^{i+1},\cdots,x^k}] F\}|(x^i)\\
&\leq \varphi((t-s)\wedge1),\ \ 1\leq i\leq k,0\leq s<t.
\end{align}
Then it follows from \eqref{DUH0}, \eqref{kb2} and the fact that $\{(P_s^{[N],N})^\ast\mu^N_0\}$ is exchangeable that
\begin{align*}
&\left|\int_{(\R^d)^k}F(x)\{(P_t^{[k],N})^\ast\mu^N_0\}(\d x)-\int_{(\R^d)^k}\{(P_{t}^\mu)^{\otimes k} F\}(x)(\mu^N_0\circ\pi_k^{-1})(\d x)\right|\\
&\leq \int_0^t \sum_{i=1}^k\int_{(\R^d)^N} \left|B^i_s(x)\right|\{(P_s^{[N],N})^\ast\mu^N_0\}(\d x)\varphi((t-s)\wedge 1)\d s\\
&= \int_0^t k\int_{(\R^d)^N} \left|B^1_s(x)\right|\{(P_s^{[N],N})^\ast\mu^N_0\}(\d x)\varphi((t-s)\wedge 1)\d s.
\end{align*}
This combined with \eqref{CON} implies that for any $t\geq 0$ and $1\leq k\leq N$,
\begin{align}\label{PTf}
\nonumber&\left|\int_{(\R^d)^k}F(x)\{(P_t^{[k],N})^\ast\mu^N_0\}(\d x)-\int_{(\R^d)^k}\{(P_{t}^\mu)^{\otimes k} F\}(x)(\mu^N_0\circ\pi_k^{-1})(\d x)\right|\\
&\leq \int_0^t\varphi(s\wedge 1)\d s g(t)\left\{\frac{k}{N}\widetilde{\W}_1(\mu^N_0,\mu_0^{\otimes N})+k\Delta(N)\{1+\{\mu_0(|\cdot|^{p})\}^{\frac{1}{p}}\}\right\}.
    \end{align}
On the other hand, for any $\tilde{\pi}\in \mathbf{C}(\mu^N_0\circ\pi_k^{-1},\mu_0^{\otimes k})$, we conclude
\begin{align*}&\left|\int_{(\R^d)^k}\{(P_{t}^\mu)^{\otimes k} F\}(x)(\mu^N_0\circ\pi_k^{-1})(\d x)-\int_{(\R^d)^k}\{(P^\mu_{t})^{\otimes k}F\}(x)\mu_0^{\otimes k}(\d x)\right|\\
&\leq \int_{(\R^d)^k\times(\R^d)^k}|\{(P^\mu_{t})^{\otimes k}F\}(x)-\{(P^\mu_{t})^{\otimes k}F\}(y)|\tilde{\pi}(\d x,\d y).
\end{align*}
This together with \eqref{kb2} and
$$\widetilde{\W}_1(\mu^N_0\circ\pi_k^{-1},\mu_0^{\otimes k})\leq \frac{k}{N}\widetilde{\W}_1(\mu^N_0,\mu_0^{\otimes N})$$
implies that
\begin{align*}&\left|\int_{(\R^d)^k}\{(P_{t}^\mu)^{\otimes k} F\}(x)(\mu^N_0\circ\pi_k^{-1})(\d x)-\int_{(\R^d)^k}\{(P^\mu_{t})^{\otimes k}F\}(x)\mu_0^{\otimes k}(\d x)\right|\\
&\leq \varphi(t\wedge 1)\widetilde{\W}_1(\mu^N_0\circ\pi_k^{-1},\mu_0^{\otimes k})\leq \varphi(t\wedge1) \frac{k}{N}\widetilde{\W}_1(\mu^N_0,\mu_0^{\otimes N}).
\end{align*}
Finally, it follows from \eqref{PTf} as well as the triangle inequality that
\begin{align}\label{vsh}\nonumber&\|(P_{t}^{[k],N})^\ast\mu_0^N-(P_{t}^\ast\mu_0)^{\otimes k}\|_{var}\\
&\leq \left(\int_0^t\varphi(s\wedge1)\d s g(t)+\varphi(t\wedge1)\right)\frac{k}{N}\widetilde{\W}_1(\mu_0^N,\mu_0^{\otimes N})\\
\nonumber&+\int_0^t\varphi(s\wedge1)\d s g(t)k\Delta(N)\{1+\{\mu_0(|\cdot|^{p})\}^{\frac{1}{p}}\}.
\end{align}
By the definition of $(P_{t}^{[k],N})^\ast$ and $P_t^\ast$, we derive from \eqref{margi} that  \begin{align*}&(P_{t+s}^{[k],N})^\ast\mu_0^N=\{(P_{t+s}^{[N],N})^\ast\mu_0^N\}\circ\pi_k^{-1}= \left\{(P_{s}^{[N],N})^\ast\{(P_{t}^{[N],N})^\ast\mu_0^N\}\right\}\circ\pi_k^{-1}\\
&\qquad\qquad\qquad\qquad\qquad\qquad\quad= (P_{s}^{[k],N})^{\ast}\{(P_{t}^{[N],N})^\ast\mu_0^N\},\\
&\ \ P_{t+s}^\ast\mu_0=P_s^\ast P_{t}^\ast\mu_0,\ \ 1\leq k\leq N, s\geq 0,t\geq 0.
\end{align*}
Then for any $t>1$, we derive from \eqref{vsh} for $t=1$ that
\begin{align*}&\|(P_{t}^{[k],N})^\ast\mu_0^N-(P_{t}^\ast\mu_0)^{\otimes k}\|_{var}=\|(P_{1}^{[k],N})^{\ast}\{(P_{t-1}^{[N],N})^\ast\mu_0^N\}-(P_{1}^\ast P_{t-1}^\ast\mu_0)^{\otimes k}\|_{var}\\
&\leq \left(\int_0^1\varphi(s)\d s g(1)+\varphi(1)\right)\frac{k}{N}\widetilde{\W}_1((P_{t-1}^{[N],N})^\ast\mu_0^N,(P_{t-1}^\ast\mu_0)^{\otimes N})\\
&+\int_0^1\varphi(s)\d s g(1)k\Delta(N)\{1+\{[P_{t-1}^\ast\mu_0](|\cdot|^{p})\}^{\frac{1}{p}}\}.
\end{align*}
This combined with \eqref{CMtty} and \eqref{unf} gives \eqref{CMY}.
\end{proof}
\section{Application in Brownian motion noise case}
In \eqref{al1} and \eqref{al}, let $$Z_t^i=(W_t^i,B_t^i),\ \ i\geq 1,$$ where $\{W_t^i\}_{i\geq 1}$ are independent $d$-dimensional Brownian motions, $\{ B_t^i\}_{i\geq1}$ are independent $n$-dimensional Brownian motions and $\{W_t^i\}_{i\geq 1}$ is independent of $\{B_t^i\}_{i\geq1}$. Let $\beta>0$,  $b^{(0)}$ and $b^{(1)}$ be defined in Section 2 and $\sigma:\R^d\to\R^d\otimes\R^{n}$ be measurable and bounded on bounded set. Consider
\begin{align} \label{al10}
\nonumber\d X^{i,N}_t&=b^{(0)}(X_t^{i,N})\d t+\frac{1}{N}\sum_{m=1}^Nb^{(1)}(X_t^{i,N},X_t^{m,N})\d t\\
&\qquad\quad+\sqrt{\beta}\d W_t^i+\sigma(X_t^{i,N})\d B_t^i,\ \ 1\leq i\leq N,
\end{align}
and independent McKean-Vlasov SDEs:
\begin{equation}\label{al0}
 \d X_t^i=b^{(0)}(X_t^i)\d t+\int_{\R^d}b^{(1)}(X_t^i,y)\L_{X_t^i}(\d y)\d t+\sqrt{\beta}\d W_t^i+\sigma(X_t^{i})\d B_t^i, \ \ 1\leq i\leq N.
\end{equation}
To derive the long time $\W_0$-$\widetilde{\W}_1$ type propagation of chaos, we make the following assumptions.
\begin{enumerate}
\item[{\bf(A)}] There exists a constant $K_\sigma>0$ such that
\begin{align}\label{bslipsg}
&\frac{1}{2}\|\sigma(x_1)-\sigma(x_2)\|^2_{HS}\leq K_\sigma|x_1-x_2|^2, \ \ x_1,x_2\in\R^d.
\end{align}
$b^{(0)}$ is continuous and there exist $R>0$, $K_1\geq 0, K_2>0$  such that
\begin{align}\label{pdi}
&\langle x_1-x_2, b^{(0)}(x_1)-b^{(0)}(x_2)\rangle\leq \gamma(|x_1-x_2|)|x_1-x_2|
\end{align}
with
$$\gamma(r)=\left\{
  \begin{array}{ll}
K_1r, & \hbox{$r\leq R$;} \\
    \left\{-\frac{K_1+K_2}{R}(r-R)+K_1\right\}r, & \hbox{$R\leq r\leq 2R$;} \\
    -K_2r, & \hbox{$r>2R$.}
  \end{array}
\right.
$$
Moreover, there exists $K_b\geq0$ such that
\begin{align}\label{lip1a}|b^{(1)}(x,y)-b^{(1)}(\tilde{x},\tilde{y})|\leq K_b(|x-\tilde{x}|+|y-\tilde{y}|),\ \ x,\tilde{x},y,\tilde{y}\in\R^d.
\end{align}
\end{enumerate}
\begin{rem}\label{mtd} Under {\bf (A)}, \eqref{al10} and \eqref{al0} are well-posed.
By \eqref{pdi}, $b^{(0)}$ is dissipative in long distance and it is equivalent to that there exist $C_1\geq0, C_2>0,r_0>0$ such that for any $x_1,x_2\in\R^d$,
\begin{align*}
&\langle x_1-x_2, b^{(0)}(x_1)-b^{(0)}(x_2)\rangle\leq C_1|x_1-x_2|^2 1_{\{|x_1-x_2|\leq r_0\}}-C_2|x_1-x_2|^21_{\{|x_1-x_2|>r_0\}}.
\end{align*}
\end{rem}

Now, we are in the position to give the main result and provide its proof.
\begin{thm}\label{POC10}
Assume {\bf(A)}. Let $\mu_0\in \scr P_{1+\delta}(\R^d)$ for some $\delta\in(0,1)$ and $\mu_0^N\in\scr P_1((\R^d)^N)$ be exchangeable.
If \begin{align}\label{kb-kd}
K_b<\frac{2\beta^2}{(K_2-K_\sigma)\left(\int_0^\infty s\e^{\frac{1}{2\beta}\int_0^s\{\gamma(v)+K_\sigma v\}\d v}\d s\right)^2},
 \end{align}
 then there exists a positive constant $c$ independent of $k$ and $N$ such that
\begin{align*}&\|(P_t^{[k],N})^\ast\mu^N_0-(P_t^\ast\mu_0)^{\otimes k}\|_{var}\\
&\leq kc\e^{-ct}\frac{\widetilde{\W}_1(\mu_0^N,\mu_0^{\otimes N})}{N}+c\{1+\{\mu_0(|\cdot|^{1+\delta})\}^{\frac{1}{1+\delta}}\}kN^{-\frac{\delta}{1+\delta}},\ \ 1\leq k\leq N,t\geq 1.
\end{align*}
\begin{rem} (1) Compared with the result in \cite{DEGZ}, the noise can be allowed to be multiplicative in Theorem \ref{POC10}.

(2) When $\sigma=0$, $b^{(0)},b^{(1)}\in C^1$, $b^{(0)}=\nabla V_1+\nabla V_2$, $b^{(1)}(x,y)=\nabla_x W(x,y)$, $V_1$ is $\rho$-strongly convex and the coefficients satisfy  $$\|V_2\|_\infty+\|W\|_\infty+\|\nabla_x W\|_\infty<\infty,\ \ \beta>\frac{8}{\rho}\|\nabla_x W\|_\infty\exp\{2\frac{\|V_2\|_\infty+2\|W\|_\infty}{\beta}\},$$ the density $u_0$ of $\mu_0$ with respect to the invariant probability measure of $\L_{X_t^i}$ satisfies $\log u_0=\bar{u}+\tilde{u}$ for bounded $\bar{u}$ and Lipschitz continuous $\tilde{u}$, and the initial distribution $\mu_0$ and $\mu_0^N$ have finite moments of all orders, \cite[Corollary 3.8]{MRW} derives the sharp long time entropy-entropy propagation of chaos with rate $\frac{k^2}{N^2}$, which implies the sharp long time $\W_0$-entropy type propagation of chaos \eqref{e-e1} with rate $\frac{k}{N}$. In Theorem \ref{POC10} above, $b^{(1)}$ is allowed to be unbounded, and $\mu_0\in \scr P_{1+\delta}(\R^d)$ for some $\delta\in(0,1)$ and $\mu_0^N\in\scr P_1((\R^d)^N)$. So, Theorem \ref{POC10} is not covered in \cite[Corollary 3.8]{MRW}.

(3) From the comparison above, one of the advantage of coupling method is to allow weaker conditions on initial distribution, i.e. $\mu_0^N$ can be singular with $\mu_0^{\otimes N}$ and the coefficients to be more general, whereas the cost is to reduce the rate of propagation of chaos since we need to estimate
\begin{align*}\E\left|\frac{1}{N}\sum_{m=1}^N b^{(1)}(X_s^{i},X_s^{m})-\int_{\R^d}b^{(1)}(X_s^{i},y)\L_{X_s^i}(\d y)\right|,
\end{align*}
which is provided in Lemma \ref{CTY} below, and the central limit theorem tells that the sharp rate is $N^{\frac{1}{2}}$ when $\L_{X_s^i}$ has finite second moment.
\end{rem}
\end{thm}
\begin{proof} Firstly, it follows from \eqref{tvc} that for any $\gamma^n\to \gamma$ and $\zeta^n\to\zeta$ weakly in $\scr P((\R^d)^k)$ as $n\to\infty$,
\begin{align}\label{upl}\nonumber\|\gamma-\zeta\|_{var}=\sup_{|f|\leq 1,f\in C_b((\R^d)^k)}|\gamma(f)-\zeta(f)|&=\sup_{|f|\leq 1,f\in C_b((\R^d)^k)}\lim_{n\to\infty}|\gamma^n(f)-\zeta^n(f)|\\
&\leq \liminf_{n\to\infty}\|\gamma^n-\zeta^n\|_{var}.
\end{align}
Combining \eqref{upl} with Lemma \ref{yap} below, we may and do
assume that there exists a constant $K_0>0$ such that
\begin{align}\label{bslip}
&|b^{(0)}(x_1)-b^{(0)}(x_2)|\leq K_0|x_1-x_2|, \ \ x_1,x_2\in\R^d.
\end{align}

Next, we intend to verify the conditions (i)-(iv) in Theorem \ref{POCin} one by one.

(1)
Take $\F_0$-measurable random variables $(X_0^{i,N})_{1\leq i\leq N}$ and $(X_0^{i})_{1\leq i\leq N}$ such that $$\L_{(X_0^{i,N})_{1\leq i\leq N}}=\mu_0^N,\ \ \L_{(X_0^{i})_{1\leq i\leq N}}=\mu_0^{\otimes N}.$$
Fix $t>0$. Recall $P_{s,t}^\mu$ and $(P_{s,t}^\mu)^{\otimes k}$ are defined in \eqref{Pmt} and \eqref{Tenso} respectively. By \eqref{bslip}, \eqref{bslipsg}, \eqref{lip1a} and the fact that $\beta\neq0$, the backward Kolmogorov equation
\begin{align}\label{BKE}\frac{\d P_{s,t}^\mu f}{\d s}=-\scr L_s^\mu P^\mu_{s,t}f,\ \ f\in C_b^2(\R^d),\|f\|_\infty\leq1, s\in[0,t]
\end{align}
holds, here
$$\scr L_s^\mu=\<b^{(0)},\nabla\>+\left\<\int_{\R^d}b^{(1)}(\cdot,y)\mu_s(\d y),\nabla\right\>+\frac{1}{2}\mathrm{Tr}[(\beta I_{d\times d}+\sigma\sigma^\ast)\nabla^2].$$
Recall that for any $F\in C^1((\R^d)^k)$, $1\leq i\leq k$ and $x=(x^1,x^2,\cdots,x^k)\in(\R^d)^k$, $\nabla_i F(x)$ represents the gradient with respect to $x^i$. Simply denote $\nabla_i^2=\nabla_i\nabla_i$.

Next, we fix $F\in C_b^2((\R^d)^k)$ with $\|F\|_\infty\leq 1$. Define
\begin{align*}(\scr L_s^\mu)^{i}F(x)
&=\<b^{(0)}(x^i),\nabla _iF(x)\>+\left\<\int_{\R^d}b^{(1)}(x^i,y)\mu_s(\d y),\nabla_i F(x)\right\>\\
&\quad+\frac{1}{2}\mathrm{Tr}[(\beta I_{d\times d}+(\sigma\sigma^\ast)(x^i))\nabla^2_iF(x)],\ \ x=(x^1,x^2,\cdots,x^k)\in(\R^{d})^k,1\leq i\leq k,
\end{align*}
and
\begin{align*}(\scr L_s^\mu)^{\otimes k}F(x)
&=\sum_{i=1}^k(\scr L_s^\mu)^{i}F(x),\ \ x=(x^1,x^2,\cdots,x^k)\in(\R^{d})^k.
\end{align*}
 We now claim that
\begin{align}\label{KOL}\frac{\d (P_{s,t}^\mu)^{\otimes k} F}{\d s}=-(\scr L_s^\mu)^{\otimes k} (P^\mu_{s,t})^{\otimes k}F, \ \ s\in[0,t].
\end{align}
In fact, for any $(s_1,s_2,\cdots,s_k)\in[0,t]^k$ and $x=(x^1,x^2,\cdots,x^k)\in(\R^{d})^k$, define
 $$\Psi_F(s_1,s_2,\cdots,s _k,x)=(P_{s_1,t}^\mu\otimes P_{s_2,t}^\mu\otimes\cdots\otimes P_{s_k,t}^\mu) F(x),$$
and
\begin{align*}&\{[\scr T_{s_1,s_2,\cdots,s_{i-1},s_{i+1},\cdots,s_k,t}^{x^1,x^2,\cdots,x^{i-1},x^{i+1},\cdots,x^k}] F\}(z)\\
&=\E F(X_{s_1,t}^{1,\mu,x^1},X_{s_2,t}^{2,\mu,x^2},\cdots,X_{s_{i-1},t}^{i-1,\mu,x^{i-1}}, z,X_{s_{i+1},t}^{i+1,\mu,x^{i+1}}\cdots,X_{s_k,t}^{k,\mu,x^k}),\ \ z\in\R^d, 1\leq i\leq k,
\end{align*}
where $(P_{s_1,t}^\mu\otimes P_{s_2,t}^\mu\otimes\cdots\otimes P_{s_k,t}^\mu)$ is given in \eqref{Tenso1}.
Then it follows from Fubini's theorem that
$$\Psi_F(s_1,s_2,\cdots,s _k,x)=P_{s_i,t}^\mu\{[\scr T_{s_1,s_2,\cdots,s_{i-1},s_{i+1},\cdots,s_k,t}^{x^1,x^2,\cdots,x^{i-1},x^{i+1},\cdots,x^k}] F\}(x^i),\ \ 1\leq i\leq k.$$
Combining this with \eqref{BKE}, the definition of $(\scr L_{s_i}^\mu)^{i}$ and Fubini's theorem, we conclude that
$$\frac{\partial}{\partial s_i}\Psi_F(s_1,s_2,\cdots,s _k,x)=-(\scr L_{s_i}^\mu)^{i}(P_{s_1,t}^\mu\otimes P_{s_2,t}^\mu\otimes\cdots\otimes P_{s_k,t}^\mu) F,\ \ 1\leq i\leq k, $$
which together with the fact $(P_{s,t}^\mu)^{\otimes k} F(x)=\Psi_F(s,s,\cdots,s,x)$ and the definition of $(\scr L_s^\mu)^{\otimes k}$ yields  \eqref{KOL}.

Now, we are in the position to prove condition (i) in Theorem \ref{POCin}. Recall that $$B^i_s(x)=\frac{1}{N}\sum_{m=1}^Nb^{(1)}(x^i,x^m) -\int_{\R^d}b^{(1)}(x^i,y)\mu_s(\d y),\ \ x=(x^1,x^2,\cdots,x^N)\in(\R^d)^N.$$
Combining \eqref{KOL} with It\^{o}'s formula, for any $s\in[0,t]$, we have
\begin{align*}
&\d [(P_{s,t}^\mu)^{\otimes k} F](X_s^{1,N},X_s^{2,N},\cdots,X_s^{k,N})\\
&=\left[-(\scr L_s^\mu)^{\otimes k} (P^\mu_{s,t})^{\otimes k}F\right](X_s^{1,N},X_s^{2,N},\cdots,X_s^{k,N})\d s\\
&+\sum_{i=1}^k\<b^{(0)}(X_s^{i,N}),\nabla _i[(P_{s,t}^\mu)^{\otimes k} F](X_s^{1,N},X_s^{2,N},\cdots,X_s^{k,N})\>\d s\\
&+\sum_{i=1}^k\left\<\frac{1}{N}\sum_{m=1}^Nb^{(1)}(X_s^{i,N},X_s^{m,N}),\nabla_i [(P_{s,t}^\mu)^{\otimes k} F](X_s^{1,N},X_s^{2,N},\cdots,X_s^{k,N})\right\>\d s\\
&+\frac{1}{2}\sum_{i=1}^k\mathrm{Tr}[(\beta I_{d\times d}+(\sigma\sigma^\ast)(X_s^{i,N}))\nabla^2_i[(P_{s,t}^\mu)^{\otimes k} F](X_s^{1,N},X_s^{2,N},\cdots,X_s^{k,N})]\d s+\d M_s\\
&=\sum_{i=1}^k\left\<B_s^i(X_s^{1,N},X_s^{2,N},\cdots,X_s^{N,N}),[\nabla_{i}(P^\mu_{s,t})^{\otimes k}F](X_s^{1,N},X_s^{2,N},\cdots,X_s^{k,N})\right\>\d s+\d M_s
\end{align*}
for some martingale $M_s$.
Integrating with respect to $s$ from $0$ to $t$ and taking expectation, we arrive at
\begin{align}\label{DUH}
\nonumber&\int_{(\R^d)^k}F(x)\{(P_t^{[k],N})^\ast\mu^N_0\}(\d x)-\int_{(\R^d)^k}\{(P_{t}^\mu)^{\otimes k} F\}(x)(\mu^N_0\circ\pi_k^{-1})(\d x)\\
&=\int_0^t\sum_{i=1}^k\int_{(\R^d)^N}\bigg\<B^i_s(x),[\nabla_{i}(P^\mu_{s,t})^{\otimes k}F](\pi_k(x))\bigg\>\{(P_s^{[N],N})^\ast\mu^N_0\}(\d x)\d s.
    \end{align}
The condition (i) in Theorem \ref{POCin} follows.

(2) By \eqref{bslipsg}-\eqref{lip1a} and \cite[Theorem 4.1 (1)]{W18} for $\kappa_2(t)=0$, there exists a constant $c_0>0$ independent of $K_0$ such that
\begin{align}\label{gra}|\nabla P^\mu_{r,t} f|\leq \frac{c_0}{(t-r)^{1/2}\wedge1}\|f\|_\infty,\ \ 0\leq r<t, f\in \scr B_b(\R^d).
\end{align}
One can also refer to \cite[Corollary 3.5]{PW} for \eqref{gra} in the time homogeneous case and \cite[Theorem 1.1 (1)]{W11} for log-Harnack inequality. Hence, condition (ii) in Theorem \ref{POCin} holds.

(3) Firstly, \eqref{pdi} implies that
\begin{align}\label{pdia}
&\langle x_1-x_2, b^{(0)}(x_1)-b^{(0)}(x_2)\rangle\leq K_1|x_1-x_2|^2,\ \ x_1,x_2\in\R^d.
\end{align}
Firstly, it is standard to derive from \eqref{bslipsg}, \eqref{lip1a} and \eqref{pdia} that
\begin{align}\label{mot}
\E((1+|X_t^1|^2)^{\frac{1+\delta}{2}})\leq c_0(t)\mu_0(1+|\cdot|^{1+\delta}),\ \ t\geq 0
\end{align}
for some increasing function $c_0:[0,\infty)\to[0,\infty)$.
Let $Z_t^{i,N}=X_t^{i}-X_t^{i,N}$. By the It\^{o}-Tanaka formula, or equivalently using $\psi_{\varepsilon}(x)=\sqrt{x^2+\varepsilon}$ to approximate $|x|$ as $\varepsilon\to0$, \eqref{bslipsg}, \eqref{lip1a} and \eqref{pdia}, we derive
\begin{align*}
\d |Z_t^{i,N}|&\leq \left\<b^{0}(X_t^{i})-b^{0}(X_t^{i,N}),\frac{Z_t^{i,N}}{|Z_t^{i,N}|}1_{\{|Z_t^{i,N}|\neq 0\}}\right\>\d t\\
&+\frac{1}{2}\|\sigma(X_t^{i,N})-\sigma(X_t^{i})\|_{HS}^2\frac{1}{|Z_t^{i,N}|}1_{\{|Z_t^{i,N}|\neq 0\}}\d t\\
&+K_b|Z_t^{i,N}|\d t+\frac{1}{N}\sum_{m=1}^N K_b|Z_t^{m,N}|\d t\\
&+\left|\frac{1}{N}\sum_{m=1}^N b^{(1)}(X_t^{i},X_t^{m})-\int_{\R^d}b^{(1)}(X_t^{i},y)\mu_t(\d y)\right|\d t\\
&+\left\<[\sigma(X_t^{i})-\sigma(X_t^{i,N})]\d B_t^i,\frac{Z_t^{i,N}}{|Z_t^{i,N}|}1_{\{|Z_t^{i,N}|\neq 0\}}\right\>\\
&\leq K_1|Z_t^{i,N}|\d t+K_b|Z_t^{i,N}|\d t+K_\sigma|Z_t^{i,N}|\d t+\frac{1}{N}\sum_{m=1}^N K_b|Z_t^{m,N}|\d t\\
&+\left|\frac{1}{N}\sum_{m=1}^N b^{(1)}(X_t^{i},X_t^{m})-\int_{\R^d}b^{(1)}(X_t^{i},y)\mu_t(\d y)\right|\d t\\
&+\left\<[\sigma(X_t^{i})-\sigma(X_t^{i,N})]\d B_t^i,\frac{Z_t^{i,N}}{|Z_t^{i,N}|}1_{\{|Z_t^{i,N}|\neq 0\}}\right\>,
\end{align*}
where we used the fact \begin{align*}&\left\<\int_{\R^d}b^{(1)}(X_t^{i},y)\mu_t(\d y)-\frac{1}{N}\sum_{m=1}^N b^{(1)}(X_t^{i,N},X_t^{m,N}), \frac{Z_t^{i,N}}{|Z_t^{i,N}|}1_{\{|Z_t^{i,N}|\neq 0\}}\right\>\\
&\leq \left|\frac{1}{N}\sum_{m=1}^N b^{(1)}(X_t^{i,N},X_t^{m,N})-\frac{1}{N}\sum_{m=1}^N b^{(1)}(X_t^{i},X_t^{m})\right|\\
&+\left|\frac{1}{N}\sum_{m=1}^N b^{(1)}(X_t^{i},X_t^{m})-\int_{\R^d}b^{(1)}(X_t^{i},y)\mu_t(\d y)\right|\\
&\leq K_b|Z_t^{i,N}|+\frac{1}{N}\sum_{m=1}^N K_b|Z_t^{m,N}|+\left|\frac{1}{N}\sum_{m=1}^N b^{(1)}(X_t^{i},X_t^{m})-\int_{\R^d}b^{(1)}(X_t^{i},y)\mu_t(\d y)\right|.
\end{align*}
Moreover, Lemma \ref{CTY} below and \eqref{mot} imply that we can find an increasing function $c:[0,\infty)\to[0,\infty)$ such that
\begin{align}\label{LAR}\left|\frac{1}{N}\sum_{m=1}^N b^{(1)}(X_t^{i},X_t^{m})-\int_{\R^d}b^{(1)}(X_t^{i},y)\mu_t(\d y)\right|\leq c(t)\{1+\{\mu_0(|\cdot|^{1+\delta})\}^{\frac{1}{1+\delta}}\}N^{-\frac{\delta}{1+\delta}}.
\end{align}
Applying Gronwall's inequality and \eqref{LAR}, we get
\begin{align}\label{myt}
\nonumber&\sum_{i=1}^N\E|Z_s^{i,N}|\\
&\leq \e^{(K_1+2K_b+K_\sigma)s}\sum_{i=1}^N\E|Z_0^{i,N}|+ \e^{(K_1+2K_b+K_\sigma)s}sc(s)\{1+\{\mu_0(|\cdot|^{1+\delta})\}^{\frac{1}{1+\delta}}\}NN^{-\frac{\delta}{1+\delta}}.
\end{align}
Since $\{X_s^{i,N}\}_{i=1}^N$ are exchangeable, we derive from \eqref{myt} and \eqref{LAR} that
\begin{align*}
&\int_{(\R^d)^N} \left|B^1_s(x)\right|\{(P_s^{[N],N})^\ast\mu^N_0\}(\d x)\\
&=\frac{1}{N}\sum_{i=1}^N\E \left|B^i_s(X_s^{1,N}, X_s^{2,N},\cdots,X_s^{N,N})\right|\\
&=\frac{1}{N}\sum_{i=1}^N\E\left|\frac{1}{N}\sum_{m=1}^N b^{(1)}(X_s^{i,N},X_s^{m,N})-\int_{\R^d}b^{(1)}(X_s^{i,N},y)\mu_s(\d y)\right|\\
&\leq\frac{1}{N}\sum_{i=1}^N\E \bigg|\frac{1}{N}\sum_{m=1}^N b^{(1)}(X_s^{i,N},X_s^{m,N})-\int_{\R^d}b^{(1)}(X_s^{i,N},y)\mu_s(\d y)\\
&\qquad\qquad\quad-\bigg(\frac{1}{N}\sum_{m=1}^N b^{(1)}(X_s^{i},X_s^{m})-\int_{\R^d}b^{(1)}(X_s^{i},y)\mu_s(\d y)\bigg)\bigg|\\
&+\frac{1}{N}\sum_{i=1}^N\E\left|\frac{1}{N}\sum_{m=1}^N b^{(1)}(X_s^{i},X_s^{m})-\int_{\R^d}b^{(1)}(X_s^{i},y)\mu_s(\d y)\right|\\
&\leq 3K_b\frac{1}{N}\sum_{i=1}^N\E|X_s^{i,N}-X_s^{i}|+ c(s)\{1+\{\mu_0(|\cdot|^{1+\delta})\}^{\frac{1}{1+\delta}}\}N^{-\frac{\delta}{1+\delta}}\\
&\leq 3K_b\e^{(K_1+2K_b+K_\sigma)s}\frac{1}{N}\sum_{i=1}^N\E|Z_0^{i,N}|\\
&+ \{3K_b\e^{(K_1+2K_b+K_\sigma)s}sc(s) +c(s)\}\{1+\{\mu_0(|\cdot|^{1+\delta})\}^{\frac{1}{1+\delta}}\}N^{-\frac{\delta}{1+\delta}}.
\end{align*}
Letting $$g(s)=\max\{3K_b\e^{(K_1+2K_b+K_\sigma)s},3K_b\e^{(K_1+2K_b+K_\sigma)s}sc(s) +c(s)\},$$  and taking infimum with respect to $(X_0^{i,N},X_0^i)_{1\leq i\leq N}$ with $\L_{(X_0^{i,N})_{1\leq i\leq N}}=\mu_0^N, \L_{(X_0^{i})_{1\leq i\leq N}}=\mu_0^{\otimes N}$, we get
\begin{align}\label{reu}\nonumber&\int_{(\R^d)^N} \left|B^1_s(x)\right|\{(P_s^{[N],N})^\ast\mu^N_0\}(\d x)\\
&\leq g(s)\left\{\frac{1}{N}\widetilde{\W}_1(\mu^N_0,\mu_0^{\otimes N})+\{1+\{\mu_0(|\cdot|^{1+\delta})\}^{\frac{1}{1+\delta}}\}N^{-\frac{\delta}{1+\delta}}\right\}.
\end{align}
Therefore, we obtain condition (iii) in Theorem \ref{POCin}.

(4) To verify condition (iv) in Theorem \ref{POCin} under {\bf (A)}, we adopt the technique of asymptotic reflection coupling. For any $\varepsilon\in(0,1]$, let $\pi_R^\varepsilon\in[0,1]$ and $\pi_S^\varepsilon$ be two Lipschitz continuous function on $[0,\infty)$ satisfying
\begin{align}\label{pir}\pi_R^\varepsilon(x)=\left\{
      \begin{array}{ll}
        1, & \hbox{$x\geq \varepsilon$;} \\
        0, & \hbox{$x\leq \frac{\varepsilon}{2}$}
      \end{array}
    \right.,\ \ (\pi_R^\varepsilon)^2+(\pi_S^\varepsilon)^2=1.
\end{align}
Let $\{\tilde{W}^i_t\}_{i\geq 1}$ be independent Brownian motions and independent of $\{W_t^i,B_t^i\}_{i\geq 1}$. Construct
\begin{equation*}\begin{split}
\d \tilde{X}_t^{i}&=b^{(0)}(\tilde{X}_t^i)\d t+\int_{\R^d}b^{(1)}(\tilde{X}_t^i,y)\mu_t(\d y)\d t\\
&+\sqrt{\beta}\pi_R^\varepsilon(|\tilde{Z}_t^{i,N}|)\d
W_t^i+\sqrt{\beta}\pi_S^\varepsilon(|\tilde{Z}_t^{i,N}|)\d \tilde{W}^i_t+\sigma(\tilde{X}_t^{i})\d B_t^i,
\end{split}\end{equation*}
and
\begin{equation*}\begin{split}
\d \tilde{X}_t^{i,N}&=b^{(0)}(\tilde{X}_t^{i,N})\d t+\frac{1}{N}\sum_{m=1}^Nb^{(1)}(\tilde{X}_t^{i,N},\tilde{X}_t^{m,N})\d t\\
&+\sqrt{\beta}\pi_R^\varepsilon(|\tilde{Z}_t^{i,N}|)(I_{d\times d}-2\tilde{U}_t^{i,N}\otimes \tilde{U}_t^{i,N})\d
W_t^i\\
&+\sqrt{\beta}\pi_S^\varepsilon(|\tilde{Z}_t^{i,N}|)\d \tilde{W}^i_t
+\sigma(\tilde{X}_t^{i,N})\d B_t^i,
\end{split}\end{equation*}
where  $\tilde{Z}_t^{i,N}=\tilde{X}_t^i-\tilde{X}_t^{i,N}$, $\tilde{U}_t^{i,N}=\frac{\tilde{Z}_t^{i,N}}{|\tilde{Z}_t^{i,N}|}1_{\{|\tilde{Z}_t^{i,N}|\neq 0\}}$ and $\L_{(\tilde{X}_0^{i,N})_{1\leq i\leq N}}=\mu_0^N$ and $\L_{(\tilde{X}_0^{i})_{1\leq i\leq N}}=\mu_0^{\otimes N}$.
By the It\^{o}-Tanaka formula, \eqref{bslipsg}, \eqref{pdi} and \eqref{lip1a}, we have
\begin{align*}
\d |\tilde{Z}_t^{i,N}|&\leq \left\<b^{0}(\tilde{X}_t^{i})-b^{0}(\tilde{X}_t^{i,N}),\frac{\tilde{Z}_t^{i,N}} {|\tilde{Z}_t^{i,N}|}1_{\{|\tilde{Z}_t^{i,N}|\neq 0\}}\right\>\d t \\ &+\frac{1}{2}\|\sigma(\tilde{X}_t^{i,N})-\sigma(\tilde{X}_t^{i})\|_{HS}^2\frac{1}{|\tilde{Z}_t^{i,N}|}1_{\{|\tilde{Z}_t^{i,N}|\neq 0\}}\d t\\
&+K_b|\tilde{Z}_t^{i,N}|\d t+\frac{1}{N}\sum_{m=1}^N K_b|\tilde{Z}_t^{m,N}|\d t\\
&+\left|\frac{1}{N}\sum_{m=1}^N b^{(1)}(\tilde{X}_t^{i},\tilde{X}_t^{m})-\int_{\R^d}b^{(1)}(\tilde{X}_t^{i},y)\mu_t(\d y)\right|\d t\\
&+\left\<[\sigma(\tilde{X}_t^{i})-\sigma(\tilde{X}_t^{i,N})]\d B_t^i,\frac{\tilde{Z}_t^{i,N}}{|\tilde{Z}_t^{i,N}|}1_{\{|\tilde{Z}_t^{i,N}|\neq 0\}}\right\>\\
&+ 2\sqrt{\beta}\pi_R^\varepsilon(|\tilde{Z}_t^{i,N}|)\left\<\frac{\tilde{Z}_t^{i,N}}{|\tilde{Z}_t^{i,N}|}1_{\{|\tilde{Z}_t^{i,N}|\neq 0\}},\d W_t^i\right\>\\
&\leq \gamma(|\tilde{Z}_t^{i,N}|)\d t+K_b|\tilde{Z}_t^{i,N}|\d t+K_\sigma|\tilde{Z}_t^{i,N}|\d t+\frac{1}{N}\sum_{m=1}^N K_b|\tilde{Z}_t^{m,N}|\d t\\
&+\left|\frac{1}{N}\sum_{m=1}^N b^{(1)}(\tilde{X}_t^{i},\tilde{X}_t^{m})-\int_{\R^d}b^{(1)}(\tilde{X}_t^{i},y)\mu_t(\d y)\right|\d t\\
&+\left\<[\sigma(\tilde{X}_t^{i})-\sigma(\tilde{X}_t^{i,N})]\d B_t^i,\frac{\tilde{Z}_t^{i,N}}{|\tilde{Z}_t^{i,N}|}1_{\{|\tilde{Z}_t^{i,N}|\neq 0\}}\right\>\\
&+2\sqrt{\beta}\pi_R^\varepsilon(|\tilde{Z}_t^{i,N}|)\left\<\frac{\tilde{Z}_t^{i,N}}{|\tilde{Z}_t^{i,N}|}1_{\{|\tilde{Z}_t^{i,N}|\neq 0\}},\d W_t^i\right\>,
\end{align*}
where we used
\begin{align*}&\left\<\int_{\R^d}b^{(1)}(\tilde{X}_t^{i},y)\mu_t(\d y)-\frac{1}{N}\sum_{m=1}^N b^{(1)}(\tilde{X}_t^{i,N},\tilde{X}_t^{m,N}), \frac{\tilde{Z}_t^{i,N}}{|\tilde{Z}_t^{i,N}|}1_{\{|\tilde{Z}_t^{i,N}|\neq 0\}}\right\>\\
&\leq K_b|\tilde{Z}_t^{i,N}|+\frac{1}{N}\sum_{m=1}^N K_b|\tilde{Z}_t^{m,N}|+\left|\frac{1}{N}\sum_{m=1}^N b^{(1)}(\tilde{X}_t^{i},\tilde{X}_t^{m})-\int_{\R^d}b^{(1)}(\tilde{X}_t^{i},y)\mu_t(\d y)\right|.
\end{align*}
Let $$\tilde{\gamma}(v)=\gamma(v)+K_{\sigma}v,\ \ v\geq 0,$$ and define
$$f(r)=\int_0^r\e^{-\frac{1}{2\beta}\int_0^u\tilde{\gamma}(v)\d v}\int_u^\infty s\e^{\frac{1}{2\beta}\int_0^s\tilde{\gamma}(v)\d v}\d s\d u,\ \ r\geq 0.$$
Then one can see that
\begin{align}
\label{MY1}f'(r)=\e^{-\frac{1}{2\beta}\int_0^r\tilde{\gamma}(v)\d v}\int_r^\infty s\e^{\frac{1}{2\beta}\int_0^s\tilde{\gamma}(v)\d v}\d s>0,
\end{align}
and
\begin{align}\label{sed}f''(r)=-\frac{1}{2\beta}\tilde{\gamma}(r)f'(r)-r.
\end{align}
Moreover, noting that
$\gamma(v)\geq -K_2 v$, we derive
$$\int_0^\infty s\e^{\frac{1}{2\beta}\int_0^s\{\gamma(v)+K_\sigma v\}\d v}\d s\geq \int_0^\infty s\e^{\frac{-(K_2-K_\sigma)s^2}{4\beta}}\d s=\frac{2\beta}{K_2-K_\sigma}.$$
Combining this with \eqref{kb-kd}, we have
\begin{align}\label{CBK}K_b<\frac{2\beta^2}{(K_2-K_\sigma)\left(\int_0^\infty s\e^{\frac{1}{2\beta}\int_0^s\{\gamma(v)+K_\sigma v\}\d v}\d s\right)^2}\leq \frac{K_2-K_\sigma}{2}.
\end{align}
Recall that
$$\gamma(r)=\left\{
  \begin{array}{ll}
K_1r, & \hbox{$r\leq R$;} \\
    \{-\frac{K_1+K_2}{R}(r-R)+K_1\}r, & \hbox{$R\leq r\leq 2R$;} \\
    -K_2r, & \hbox{$r>2R$}
  \end{array}
\right.
$$
for $K_2>K_b+K_\sigma$ due to \eqref{CBK}. Letting $\ell_0=\{1+\frac{K_1+K_\sigma}{K_1+K_2}\}R$, it is not difficult to see that
$$\left\{
  \begin{array}{ll}
    \tilde{\gamma}(r)\geq 0, & \hbox{$r\in [0,\ell_0]$;} \\
    \tilde{\gamma}(r)< 0, & \hbox{$r\in(\ell_0,\infty)$.}
  \end{array}
\right.
$$
By \eqref{sed} and \eqref{MY1}, we derive
\begin{align}\label{fii}f''(r)\leq 0, \ \ r\in[0,\ell_0].
\end{align}
In view of the definition of $\gamma$ and $\tilde{\gamma}$, we conclude that
$
\frac{r}{-\tilde{\gamma}(r)}
$
is decreasing in $(\ell_0,\infty)$.
This combined with the integration by parts formula gives
\begin{align*}
\int_r^\infty s\e^{\frac{1}{2\beta}\int_0^s\tilde{\gamma}(v)\d v}\d s
&=\int_r^\infty \frac{2\beta s}{\tilde{\gamma}(s)}\left(\frac{\d}{\d s} \e^{\frac{1}{2\beta}\int_0^s\tilde{\gamma}(v)\d v}\right)\d s\\
&\leq -\frac{2\beta r}{\tilde{\gamma}(r)}\e^{\frac{1}{2\beta}\int_0^r\tilde{\gamma}(v)\d v},\ \ r> \ell_0,
\end{align*}
which together with \eqref{MY1} and \eqref{sed} yields
$$f''(r)\leq 0,\ \ r\in(\ell_0,\infty).$$
This as well as \eqref{fii} means that $f''\leq 0$ so that $f(r)\leq f'(0)r$ and $\frac{f(r)}{r}$ is decreasing on $(0,\infty)$. As a result, we derive from \eqref{MY1} that
\begin{align*}
\inf_{r>0} \frac{f(r)}{r}&=\lim_{r\to\infty}\frac{f(r)}{r}=\lim_{r\to \infty}f'(r)=\lim_{r\to \infty}\frac{\int_r^\infty s\e^{\frac{1}{2\beta}\int_0^s\tilde{\gamma}(v)\d v}\d s}{\e^{\frac{1}{2\beta}\int_0^r\tilde{\gamma}(v)\d v}}=\frac{2\beta}{K_2-K_\sigma}.
\end{align*}
So, we conclude that
\begin{align}\label{cop}\frac{2\beta}{K_2-K_\sigma}r\leq f(r)\leq f'(0)r.
\end{align}
By It\^{o}'s formula and $f''\leq 0$, we have
\begin{align}\label{itf}
\nonumber\d f(|\tilde{Z}_t^{i,N}|)&\leq f'(|\tilde{Z}_t^{i,N}|)\tilde{\gamma}(|\tilde{Z}_t^{i,N}|)\d t\\
\nonumber&+f'(|\tilde{Z}_t^{i,N}|)K_b|\tilde{Z}_t^{i,N}|\d t+f'(|\tilde{Z}_t^{i,N}|)\frac{1}{N}\sum_{m=1}^N K_b|\tilde{Z}_t^{m,N}|\d t\\
&+f'(|\tilde{Z}_t^{i,N}|)\left|\frac{1}{N}\sum_{m=1}^N b^{(1)}(\tilde{X}_t^{i},\tilde{X}_t^{m})-\int_{\R^d}b^{(1)}(\tilde{X}_t^{i},y)\mu_t(\d y)\right|\d t\\
\nonumber&+f'(|\tilde{Z}_t^{i,N}|)\left\<[\sigma(\tilde{X}_t^{i})-\sigma(\tilde{X}_t^{i,N})]\d B_t^i,\frac{\tilde{Z}_t^{i,N}}{|\tilde{Z}_t^{i,N}|}1_{\{|\tilde{Z}_t^{i,N}|\neq 0\}}\right\>\\
\nonumber&+f'(|\tilde{Z}_t^{i,N}|)2\sqrt{\beta}\pi_R^\varepsilon(|\tilde{Z}_t^{i,N}|) \left\<\frac{\tilde{Z}_t^{i,N}}{|\tilde{Z}_t^{i,N}|}1_{\{|\tilde{Z}_t^{i,N}|\neq 0\}},\d W_t^i\right\>\\
\nonumber&+2\beta f''(|\tilde{Z}_t^{i,N}|)\pi_R^\varepsilon(|\tilde{Z}_t^{i,N}|)^2\d t.
\end{align}
It follows from \eqref{sed} and $\|f'\|_\infty=f'(0)$ that
\begin{align*}
&f'(|\tilde{Z}_t^{i,N}|)\tilde{\gamma}(|\tilde{Z}_t^{i,N}|)+2\beta f''(|\tilde{Z}_t^{i,N}|)\pi_R^\varepsilon(|\tilde{Z}_t^{i,N}|)^2\\
&\leq \left(f'(|\tilde{Z}_t^{i,N}|)\tilde{\gamma}(|\tilde{Z}_t^{i,N}|)+2\beta f''(|\tilde{Z}_t^{i,N}|)\right)\pi_R^\varepsilon(|\tilde{Z}_t^{i,N}|)^2\\ &+\|f'\|_\infty\left\{\sup_{s\in[0,\varepsilon]}\gamma^{+}(s)+K_\sigma\varepsilon\right\}\\
&\leq -2\beta |\tilde{Z}_t^{i,N}|+2\beta |\tilde{Z}_t^{i,N}|\left(1-\pi_R^\varepsilon(|\tilde{Z}_t^{i,N}|)^2\right) +\|f'\|_\infty\left\{\sup_{s\in[0,\varepsilon]}\gamma^{+}(s)+K_\sigma\varepsilon\right\}\\
&\leq -2\beta |\tilde{Z}_t^{i,N}|+2\beta \varepsilon +f'(0)\left\{\sup_{s\in[0,\varepsilon]}\gamma^{+}(s)+K_\sigma\varepsilon\right\}.
\end{align*}
This combined with \eqref{cop} and \eqref{itf} gives
\begin{align*}
\d \sum_{i=1}^Nf(|\tilde{Z}_t^{i,N}|)
&\leq -\left\{\frac{2\beta}{f'(0)} -f'(0)\frac{(K_2-K_\sigma)}{\beta}K_b\right\}\sum_{i=1}^Nf(|\tilde{Z}_t^{i,N}|)\d t\\
&+2\sum_{i=1}^N\beta \varepsilon \d t +f'(0)\sum_{i=1}^N\left\{\sup_{s\in[0,\varepsilon]}\gamma^{+}(s)+K_\sigma\varepsilon\right\}\d t\\
&+\sum_{i=1}^Nf'(|\tilde{Z}_t^{i,N}|)\left|\frac{1}{N}\sum_{m=1}^N b^{(1)}(\tilde{X}_t^{i},\tilde{X}_t^{m})-\int_{\R^d}b^{(1)}(\tilde{X}_t^{i},y)\mu_t(\d y)\right|\d t\\
&+\sum_{i=1}^Nf'(|\tilde{Z}_t^{i,N}|)\left\<[\sigma(\tilde{X}_t^{i})-\sigma(\tilde{X}_t^{i,N})]\d B_t^i,\frac{\tilde{Z}_t^{i,N}}{|\tilde{Z}_t^{i,N}|}1_{\{|\tilde{Z}_t^{i,N}|\neq 0\}}\right\>\\
&+\sum_{i=1}^Nf'(|\tilde{Z}_t^{i,N}|)2\sqrt{\beta}\pi_R^\varepsilon(|\tilde{Z}_t^{i,N}|) \left\<\frac{\tilde{Z}_t^{i,N}}{|\tilde{Z}_t^{i,N}|}1_{\{|\tilde{Z}_t^{i,N}|\neq 0\}},\d W_t^i\right\>.
\end{align*}
Let $\lambda=\frac{2\beta}{f'(0)} -f'(0)\frac{(K_2-K_\sigma)}{\beta}K_b$. Then \eqref{kb-kd} and the fact that $f'(0)=\int_0^\infty s\e^{\frac{1}{2\beta}\int_0^s\tilde{\gamma}(v)\d v}\d s$ imply $\lambda>0$. Hence, it follows that
\begin{align}\label{fzn}
\nonumber&\sum_{i=1}^N\E f(|\tilde{Z}_t^{i,N}|)\\
\nonumber&\leq \exp\left\{-\lambda t\right\}\sum_{i=1}^N\E f(|\tilde{Z}_0^{i,N}|)\\
&+N\int_0^t\exp\left\{-\lambda(t-s)\right\}\left\{2\beta \varepsilon  +f'(0)\left\{\sup_{s\in[0,\varepsilon]}\gamma^{+}(s)+K_\sigma\varepsilon\right\}\right\}\d s\\
\nonumber&+\int_0^t\exp\left\{-\lambda(t-s)\right\}f'(0)N\\
\nonumber&\qquad\qquad\quad\times\E\left|\frac{1}{N}\sum_{m=1}^N b^{(1)}(\tilde{X}_s^{i},\tilde{X}_s^{m})-\int_{\R^d}b^{(1)}(\tilde{X}_s^{i},y)\mu_s(\d y)\right|\d s.
\end{align}
Moreover, It\^{o}'s formula implies that
\begin{align*}
\d (1+|\tilde{X}_t^{i}|^2)&=2\<\tilde{X}_t^{i},b^{(0)}(\tilde{X}_t^{i})\>\d t+2\left\<\tilde{X}_t^{i},\int_{\R^d}b^{(1)}(\tilde{X}_t^{i},y)\mu_t(\d y)\right\>\d t\\
&+\beta d\d t+\|\sigma(\tilde{X}_t^{i})\|_{HS}^2\d t+\d \tilde{M}_t,~t\ge0
\end{align*}
for some martingale $\tilde{M}_t$.
By {\bf(A)}, we can find a constant $C_0>0$ such that
\begin{align*}&2\<x,b^{(0)}(x)\>+2\left\<x,\int_{\R^d}b^{(1)}(x,y)\mu_t(\d y)\right\>+\beta d+\|\sigma(x)\|_{HS}^2\\
&\leq (2K_1+2K_2)|x|^21_{\{|x|\leq 2R\}}-(2K_2-2K_\sigma)|x|^2+2\<x,b^{(0)}(0)\>+\beta d\\
&+2\sqrt{2K_\sigma}\|\sigma(0)\|_{HS}|x|+\|\sigma(0)\|_{HS}^2+ 2|x|K_b(|x|+\mu_t(|\cdot|))+2|x||b^{(1)}(0,0)|\\
&\leq (2K_1+2K_2)4R^2+\beta d+\|\sigma(0)\|_{HS}^2-(2K_2-2K_\sigma-4K_b)|x|^2\\
&+2|x|(|b^{(0)}(0)|+\sqrt{2K_\sigma}\|\sigma(0)\|_{HS}+|b^{(1)}(0,0)|)-2K_b|x|^2+ 2(1+|x|^2)^{\frac{1}{2}}K_b\mu_t(|\cdot|)\\
&\leq C_0-(K_2-K_\sigma-2K_b)(1+|x|^2)-2K_b(1+|x|^2)+2(1+|x|^2)^{\frac{1}{2}}K_b\mu_t(|\cdot|)\\
&= C_0-(K_2-K_\sigma-2K_b)(1+|x|^2)\\
&+(1+|x|^2)^{\frac{1-\delta}{2}}\{-2K_b(1+|x|^2)^{\frac{1+\delta}{2}}+2(1+|x|^2)^{\frac{\delta}{2}}K_b\mu_t(|\cdot|)\}\\
&\leq C_0-(K_2-K_\sigma-2K_b)(1+|x|^2)\\
&+(1+|x|^2)^{\frac{1-\delta}{2}}2K_b\frac{1}{1+\delta}\left\{-(1+|x|^2)^{\frac{1+\delta}{2}}+ \mu_t((1+|\cdot|^2)^{\frac{1+\delta}{2}})\right\},\ \ x\in\R^d.
\end{align*}
This together with the It\^{o} formula gives
\begin{align*}
\d (1+|\tilde{X}_t^{i}|^2)^{\frac{1+\delta}{2}}&\leq C_1\d t- \frac{1+\delta}{2}(K_2-K_\sigma-2K_b)(1+|\tilde{X}_t^{i}|^2)^{\frac{1+\delta}{2}}\d t\\
&+K_b\left\{-(1+|\tilde{X}_t^{i}|^2)^{\frac{1+\delta}{2}}+ \mu_t((1+|\cdot|^2)^{\frac{1+\delta}{2}})\right\}+\d \bar{M}_t,~t\ge0.
\end{align*}
Combining this with \eqref{CBK}, we conclude that there exists a constant $c_0>0$ such that
$$\E|\tilde{X}_t^i|^{1+\delta}\leq c_0(1+\E|\tilde{X}_0^i|^{1+\delta}),\ \ t\geq 0.$$
So, we derive from Lemma \ref{CTY} below that
\begin{align}\label{ubi}\E\left|\frac{1}{N}\sum_{m=1}^N b^{(1)}(\tilde{X}_s^{i},\tilde{X}_s^{m})-\int_{\R^d}b^{(1)}(\tilde{X}_s^{i},y)\mu_s(\d y)\right|\d s\leq \tilde{c}_0\{1+\{\mu_0(|\cdot|^{1+\delta})\}^{\frac{1}{1+\delta}}\}N^{-\frac{\delta}{1+\delta}}
\end{align}
for some constant $\tilde{c}_0>0$. Note that different from \eqref{LAR}, $\tilde{c}_0$ in \eqref{ubi} is independent of $s$.
Substituting \eqref{ubi} into \eqref{fzn} and applying \eqref{cop}, we can find some constants $c_1,c_2>0$ such that
\begin{align*}
 \sum_{i=1}^N\E|\tilde{Z}_t^{i,N}|&\leq  c_1\e^{-c_2t}\sum_{i=1}^N\E|\tilde{Z}_0^{i,N}|+c_1N\{1+\{\mu_0(|\cdot|^{1+\delta})\}^{\frac{1}{1+\delta}}\}N^{-\frac{\delta}{1+\delta}}\\
&+c_1\left\{2\beta \varepsilon  +f'(0)\left\{\sup_{s\in[0,\varepsilon]}\gamma^{+}(s)+K_\sigma\varepsilon\right\}\right\}.
\end{align*}
Letting $\varepsilon\to 0$, we derive
\begin{align*}
\widetilde{\W}_1((P_t^{[N],N})^\ast\mu^N_0,(P_t^\ast\mu_0)^{\otimes N})&\leq  \sum_{i=1}^N\E|\tilde{Z}_t^{i,N}|\\
&\leq  c_1\e^{-c_2t}\sum_{i=1}^N\E|\tilde{Z}_0^{i,N}|+c_1N\{1+\{\mu_0(|\cdot|^{1+\delta})\}^{\frac{1}{1+\delta}}\}N^{-\frac{\delta}{1+\delta}}.
\end{align*}
Taking infimum with respect to $(\tilde{X}_0^{i,N},\tilde{X}_0^i)_{1\leq i\leq N}$ with $\L_{(\tilde{X}_0^{i,N})_{1\leq i\leq N}}=\mu_0^N, \L_{(\tilde{X}_0^{i})_{1\leq i\leq N}}=\mu_0^{\otimes N}$, we get
\begin{align*}
\widetilde{\W}_1((P_t^{[N],N})^\ast\mu^N_0,(P_t^\ast\mu_0)^{\otimes N})\leq  c_1\e^{-c_2t}\widetilde{\W}_1(\mu^N_0,\mu_0^{\otimes N})+c_1N\{1+\{\mu_0(|\cdot|^{1+\delta})\}^{\frac{1}{1+\delta}}\}N^{-\frac{\delta}{1+\delta}}.
\end{align*}
Therefore, condition (iv) in Theorem \ref{POCin} holds. Finally, applying Theorem \ref{POCin}, we complete the proof.
\end{proof}

\section{Application in $\alpha$-stable noise case}
Recall that $d$-dimensional rotationally invariant $\alpha$-stable process has L\'{e}vy measure
$$\nu^\alpha(\d z)=\frac{c_{d,\alpha}}{|z|^{d+\alpha}}\d z $$ for some constant $c_{d,\alpha}>0$ and the generator $-(-\Delta)^{\frac{\alpha}{2}}$ is defined by
$$-(-\Delta)^{\frac{\alpha}{2}}f(x)=\int_{\R^d-\{0\}}\{f(x+z)-f(x)-\<\nabla f(x),z\>1_{\{|z|\leq 1\}}\}\nu^{\alpha}(\d z),\ \ f\in C_b^2(\R^d), \|f\|_\infty\leq 1. $$
Let $b^{(0)}$, $b^{(1)}$ and $\{Z_t^i\}_{i\geq 1}$ be introduced in Section 2 with $n=d$ and $\sigma=I_{d\times d}$. The \eqref{al1} and \eqref{al} reduce to
\begin{align*}
\d X^{i,N}_t=b^{(0)}(X_t^{i,N})\d t+\frac{1}{N}\sum_{m=1}^Nb^{(1)}(X_t^{i,N},X_t^{m,N})\d t+\d Z_t^i,\ \ 1\leq i\leq N,
\end{align*}
and
\begin{equation*}
 \d X_t^i=b^{(0)}(X_t^i)\d t+\int_{\R^d}b^{(1)}(X_t^i,y)\L_{X_t^i}(\d y)\d t+\d Z_t^i,\ \ 1\leq i\leq N
\end{equation*}
respectively.
We make the following assumptions.
\begin{enumerate}
\item[{\bf(B1)}] The generator of $Z_t^i$ is $-(-\Delta)^{\frac{\alpha}{2}}$ for some $\alpha\in(1,2)$.
\item[{\bf(B2)}]  $b^{(0)}$ is continuous. There exist $\ell_0>0$, $K_1\geq 0, K_2>0$, $K_b\geq0$ such that
\begin{align}\label{pdic}
\nonumber&\langle x_1-x_2, b^{(0)}(x_1)-b^{(0)}(x_2)\rangle\\
&\leq K_1| x_1-x_2|^21_{\{|x_1-x_2|\leq\ell_0\}}-K_2| x_1-x_2|^21_{\{|x_1-x_2|>\ell_0\}},
\end{align}
and
    $$|b^{(1)}(x,y)-b^{(1)}(\tilde{x},\tilde{y})|\leq K_b(|x-\tilde{x}|+|y-\tilde{y}|),\ \ x,\tilde{x},y,\tilde{y}\in\R^d.$$
\end{enumerate}
Recall that for two measures $\zeta,\tilde{\zeta}$ on $\R^d$, $$\zeta\wedge\tilde{\zeta}=\zeta-(\zeta-\tilde{\zeta})^{+}.$$
Let
\begin{align}\label{Jst}
J^\alpha(s)=\inf_{x\in\R^d,|x|\leq s}(\nu^\alpha\wedge(\delta_x\ast\nu^\alpha))(\R^d),\ \ s\geq0,
\end{align}
here $$\delta_x\ast\nu^\alpha(\d z)=\nu^\alpha(\d z-x)=\frac{c_{d,\alpha}}{|z-x|^{d+\alpha}}\d z,\ \ x\in\R^d,$$
 and hence
 $$(\nu^\alpha\wedge(\delta_x\ast\nu^\alpha))(\d z)=\frac{c_{d,\alpha}}{(|z|\vee|z-x|)^{d+\alpha}}\d z,\ \ x\in\R^d.$$
By \cite[Example 1.2]{LW}, there exist constants $\kappa>0$ and $\tilde{c}_{d,\alpha}>0$ such that
\begin{align}\label{MJT}J^\alpha(s)\geq \tilde{c}_{d,\alpha}s^{-\alpha},\ \ s\in(0,\kappa].
\end{align}
In fact, for any $r>0$ and $x\in\R^d$ with $|x|=r$, it holds
\begin{align*}(\nu^\alpha\wedge(\delta_x\ast\nu^\alpha))(\R^d)&=\int_{\R^d} \frac{c_{d,\alpha}}{(|z|\vee|z-x|)^{d+\alpha}}\d z\\
&\geq \int_{|z|\leq \frac{r}{2}} \frac{c_{d,\alpha}}{(|z|\vee|z-x|)^{d+\alpha}}\d z\\
&\geq \frac{2^{d+\alpha}}{3^{d+\alpha}}c_{d,\alpha}r^{-d-\alpha}\int_{|z|\leq \frac{r}{2}} \d z\\
&=\tilde{c}_{d,\alpha}r^{-\alpha}.
\end{align*}
So, we have
\begin{align*}
J^\alpha(s)=\inf_{0\leq r\leq s}\inf_{x\in\R^d,|x|=r}(\nu^\alpha\wedge(\delta_x\ast\nu^\alpha))(\R^d)\geq \inf_{0\leq r\leq s}\tilde{c}_{d,\alpha}r^{-\alpha}=\tilde{c}_{d,\alpha}s^{-\alpha},\ \ s>0.
\end{align*}
Moreover, for any $\eta\in(0,1)$, take \begin{align}\label{cmy}\sigma_\eta(r)=\frac{\tilde{c}_{d,\alpha}(\kappa\wedge(2\ell_0))^{2-\alpha}}{2(2\ell_0)^{1+\eta}}r^\eta,\ \ r\in[0,2\ell_0],
\end{align}
here \begin{align}\label{kkt}\frac{\tilde{c}_{d,\alpha}(\kappa\wedge(2\ell_0))^{2-\alpha}}{2(2\ell_0)^{1+\eta}}= \inf_{r\in[0,2\ell_0]}\frac{\tilde{c}_{d,\alpha}(\kappa\wedge r)^{2-\alpha}}{2r^{1+\eta}}.
\end{align}
Then $\sigma_\eta\in C([0,2\ell_0])\cap C^2((0,2\ell_0])$ and it is a nondecreasing and concave function. Moreover, \eqref{MJT}, \eqref{cmy} and \eqref{kkt} imply that
\begin{align}\label{tij}\sigma_\eta(r)\leq \frac{1}{2r}J^\alpha(\kappa\wedge r)(\kappa\wedge r)^2, \ \ r\in[0,2\ell_0].
\end{align}
Let
\begin{align}\label{gfu} \nonumber&g_\eta(r)=\left(1+\frac{K_1}{K_2}\right)\int_0^r\frac{1}{\sigma_\eta(s)}\d s, \ \ r\in[0,2\ell_0],\\
& c_1=\e^{-2K_2g_\eta(2\ell_0)}.
 \end{align}
\begin{thm}\label{POC101}
Assume {\bf(B1)}-{\bf(B2)}. Let $\mu_0\in \scr P_{1+\delta}(\R^d)$ for some $\delta\in(0,\alpha-1)$ and $\mu_0^N\in\scr P_1((\R^d)^N)$ be exchangeable. If $$K_b< \frac{2c_1^2K_2}{(1+c_1)^2},$$
then there exist positive constants $c,\lambda$ such that
\begin{align}\label{CMYb}\nonumber&\|(P_t^{[k],N})^\ast\mu^N_0-(P_t^\ast\mu_0)^{\otimes k}\|_{var}\\
&\leq kc\e^{-\lambda t}\frac{\widetilde{\W}_1(\mu_0^N,\mu_0^{\otimes N})}{N}+c\{1+\{\mu_0(|\cdot|^{1+\delta})\}^{\frac{1}{1+\delta}}\}kN^{-\frac{\delta}{1+\delta}},\ \ 1\leq k\leq N,t\geq 1.
\end{align}
\end{thm}
\begin{rem} The condition $\E|Z_t^i|^2<\infty$ in \cite[Theorem 1.2]{LMW} is removed in Theorem \ref{POC101}, which is attributed to Lemma \ref{CTY} in Appendix.
\end{rem}
\begin{proof}[Proof of Theorem \ref{POC101}]
Similar to the proof of Theorem \ref{POC10}, by the Yosida approximation in \cite[part (c) of proof of Theorem 2.1]{WW} and \eqref{upl}, it is sufficient to prove \eqref{CMYb} for  Lipschitz continuous $b^{(0)}$. So, in the following, we assume that $b^{(0)}$ is Lipschitz continuous and we will verify conditions (i)-(iv) in Theorem \ref{POCin} one by one. We should remark that the proof of conditions (i)-(iii) is similar to that of Theorem \ref{POC10}.

(1) Take $(X_0^{i,N})_{1\leq i\leq N}$ and $(X_0^{i})_{1\leq i\leq N}$ such that $\L_{(X_0^{i,N})_{1\leq i\leq N}}=\mu_0^N$ and $\L_{(X_0^{i})_{1\leq i\leq N}}=\mu_0^{\otimes N}$.
Fix $t>0$. Recall that $P_{s,t}^\mu$ and $(P_{s,t}^\mu)^{\otimes k}$ are defined in \eqref{Pmt} and \eqref{Tenso} respectively. Since $b^{(0)}$ and $b^{(1)}$ are Lipschitz continuous, the backward Kolmogorov equation
\begin{align}\label{BKEa}\frac{\d P_{s,t}^\mu f}{\d s}=-\scr L_s^\mu P^\mu_{s,t}f,\ \ f\in C_b^2(\R^d), \|f\|_\infty\leq 1
\end{align}
holds, here
$$\scr L_s^\mu=\left\<b^{(0)}+\int_{\R^d}b^{(1)}(\cdot,y)\mu_s(\d y),\nabla\right\>-(-\Delta)^{\frac{\alpha}{2}}.$$
For any $F\in C_b^2((\R^d)^k)$ with $\|F\|_\infty\leq 1$, $1\leq i\leq k$, $x=(x^1,\cdots,x^k)\in(\R^d)^k$, denote $$-(-\Delta_i)^{\frac{\alpha}{2}} F(x)=\int_{\R^d}\{F(x^1,\cdots,x^i+z,\cdots,x^k)-F(x)-\<\nabla_i F(x),z\>1_{\{|z|\leq 1\}}\}\nu^\alpha(\d z), $$
and define
\begin{align*}(\scr L_s^\mu)^{\otimes k}F(x)
&=\sum_{i=1}^k\bigg\{\left\<b^{(0)}(x^i)+\int_{\R^d}b^{(1)}(x^i,y)\mu_s(\d y),\nabla_i F(x)\right\>-(-\Delta_i)^{\frac{\alpha}{2}} F(x)\bigg\}.
\end{align*}
Repeating the argument to derive \eqref{KOL} from \eqref{BKE}, it follows from \eqref{BKEa} that
\begin{align*}\frac{\d (P_{s,t}^\mu)^{\otimes k} F}{\d s}=-(\scr L_s^\mu)^{\otimes k} (P^\mu_{s,t})^{\otimes k}F,\ \ F\in C_b^2((\R^d)^k), \|F\|_\infty\leq 1, s\in[0,t].
\end{align*}
This together with the procedure to derive \eqref{DUH} from \eqref{KOL} implies (i) in Theorem \ref{POCin}.

(2) By \cite[Corollary 2.2(2)]{WW}, there exists a constant $c_0>0$ independent of the Lipschitz constant of $b^{(0)}$ such that
\begin{align*}|\nabla P^\mu_{r,t} f|\leq c_0\frac{1}{(t-r)^{1/\alpha}\wedge 1}\|f\|_\infty,\ \ 0\leq r<t, f\in\scr B_b(\R^d).
\end{align*}
This means that (ii) in Theorem \ref{POCin} holds.

(3) Note that \eqref{pdic} implies that
\begin{align}\label{pdib}
&\langle x_1-x_2, b^{(0)}(x_1)-b^{(0)}(x_2)\rangle\leq K_1|x_1-x_2|^2.
\end{align}
It is standard to derive from \eqref{pdib} and {\bf(B1)-(B2)} that
\begin{align}\label{mot1a}
\E((1+|X_t^1|^2)^{\frac{1+\delta}{2}})\leq c_0(t)\mu_0(1+|\cdot|^{1+\delta}),\ \ t\geq 0
\end{align}
for some increasing function $c_0:[0,\infty)\to[0,\infty)$.
Let $Z_t^{i,N}=X_t^{i}-X_t^{i,N}$. It follows from It\^{o}'s formula that
\begin{align*}
\d |Z_t^{i,N}|&\leq \left\<b^{0}(X_t^{i})-b^{0}(X_t^{i,N}),\frac{Z_t^{i,N}}{|Z_t^{i,N}|}1_{\{|Z_t^{i,N}|\neq 0\}}\right\>\d t+K_b|Z_t^{i,N}|\d t+\frac{1}{N}\sum_{m=1}^N K_b|Z_t^{m,N}|\d t\\
&+\left|\frac{1}{N}\sum_{m=1}^N b^{(1)}(X_t^{i},X_t^{m})-\int_{\R^d}b^{(1)}(X_t^{i},y)\mu_t(\d y)\right|\d t\\
&\leq K_1|Z_t^{i,N}|\d t+K_b|Z_t^{i,N}|\d t+\frac{1}{N}\sum_{m=1}^N K_b|Z_t^{m,N}|\d t\\
&+\left|\frac{1}{N}\sum_{m=1}^N b^{(1)}(X_t^{i},X_t^{m})-\int_{\R^d}b^{(1)}(X_t^{i},y)\mu_t(\d y)\right|\d t.
\end{align*}
Applying Gronwall's inequality, Lemma \ref{CTY} below and \eqref{mot1a}, we get
\begin{align*}
\sum_{i=1}^N\E|Z_s^{i,N}|&\leq \e^{(K_1+2K_b)s}\sum_{i=1}^N\E|Z_0^{i,N}|+ c(s)\{1+\{\mu_0(|\cdot|^{1+\delta})\}^{\frac{1}{1+\delta}}\}NN^{-\frac{\delta}{1+\delta}}
\end{align*}
for some increasing function $c:[0,\infty)\to[0,\infty)$.
By the same argument to derive \eqref{reu} from \eqref{myt}, (iii) in Theorem \ref{POCin} holds.

(4) Finally, we will adopt asymptotic refined basic coupling and modify \cite[Proof of Theorem 1.2]{LMW} to derive \eqref{CMtty}.
Let $\kappa$ be defined in \eqref{MJT} and let
\begin{align}\label{xka}(x)_\kappa= \left\{\frac{|x|\wedge\kappa}{|x|}1_{\{|x|\neq 0\}}\right\}x,\ \ x\in\R^d.
\end{align}
For simplicity, we denote
\begin{align}\label{nua}\nu_x^\alpha=\nu^\alpha\wedge(\delta_x\ast\nu^\alpha),\ \ x\in\R^d.
\end{align}
For any $x,y,u\in\R^d$, $F\in C_b^2(\R^{d}\times \R^d)$ with $\|F\|_\infty\leq 1$, define $\scr M^{x,y,u}F:\R^d\to\R$ by
\begin{align*}(\scr M^{x,y,u}F)(z)&=F(x+z,y+z+u)-F(x,y)\\
&-\<\nabla_{x}F(x,y),z\>1_{\{|z|\leq 1\}}-\<\nabla_{y}F(x,y),z+u\>1_{\{|z+u|\leq 1\}},\ \ z\in\R^d.
\end{align*}
Define the refined basic coupling operator:
\begin{align}\label{rbc}
\nonumber(\scr L_R F)(x,y)&=\frac{1}{2}\int_{\R^d}(\scr M^{x,y,(x-y)_\kappa}F)(z)\nu_{(y-x)_\kappa}^\alpha(\d z)\\
&+\frac{1}{2}\int_{\R^d}(\scr M^{x,y,(y-x)_\kappa}F)(z)\nu_{(x-y)_\kappa}^\alpha(\d z)\\
\nonumber&+\int_{\R^d}(\scr M^{x,y,0}F)(z)\left\{\nu^\alpha-\frac{1}{2}\nu_{(x-y)_\kappa}^\alpha-\frac{1}{2}\nu_{(y-x)_\kappa}^\alpha\right\}(\d z)
\end{align}
and the synchronous coupling operator:
\begin{align}\label{sys}
(\scr L_SF)(x,y)=\int_{\R^d}(\scr M^{x,y,0}F)(z)\nu^\alpha(\d z).
\end{align}
We are now in the position to construct asymptotic refined basic coupling operator. Let $\pi_R^\varepsilon$ be defined in \eqref{pir}. For any $\varepsilon>0$, define
\begin{align*}(\scr L^{\varepsilon}F)(x,y)=\pi_R^\varepsilon(|x-y|) (\scr L_R F)(x,y)+(1-\pi_R^\varepsilon(|x-y|))(\scr L_SF)(x,y).
\end{align*}

Next, we adopt the procedure as in \cite[(2.14)]{LMW} to construct coupling processes.
Let $N^i(\d t,\d z)$ be the Poisson random measure associated to $Z_t^i$. Define
$$\hat{N}^i(\d t,\d z,\d u)=N^i(\d t,\d z)1_{[0,1]}(u)\d u.$$
Let $\rho(x,\cdot)=\frac{\d \nu_x^\alpha}{\d \nu^\alpha}$, the Radon-Nikodym derivative of $\nu_x^\alpha$ with respect to $\nu^\alpha$.
Construct
\begin{equation}\label{PPY01}
\d \bar{X}_t^{i}=b^{(0)}(\bar{X}_t^{i})\d t+\int_{\R^d}b^{(1)}(\bar{X}_t^{i},y)\mu_t(\d y)\d t+\d Z_t^i,~t\ge0,
\end{equation}
and
\begin{equation}\begin{split}\label{PPT02}
\d \bar{X}_t^{i,N}&=b^{(0)}(\bar{X}_t^{i,N})\d t+\frac{1}{N}\sum_{m=1}^Nb^{(1)}(\bar{X}_t^{i,N},\bar{X}_t^{m,N})\d t+\d Z_t^i+\d Z_t^{\varepsilon,i},~t\ge0,
\end{split}\end{equation}
where
$$\d Z_t^{\varepsilon,i}=\int_{\R^d\times [0,1]}S^\varepsilon(\bar{Z}_{t-}^{i,N},z,u)\hat{N}^i(\d t,\d z,\d u)$$
with $\bar{Z}_t^{i,N}=\bar{X}_t^{i}-\bar{X}_t^{i,N}$,
\begin{align*}S^\varepsilon(x,z,u)&=(x)_\kappa1_{\{0\leq u\leq \frac{1}{2}\pi_R^\varepsilon (|x|)\rho(-(x)_\kappa,z)\}}\\
&-(x)_\kappa1_{\{\frac{1}{2}\pi_R^\varepsilon (|x|)\rho(-(x)_\kappa,z)\leq u\leq \frac{1}{2}\pi_R^\varepsilon (|x|)\rho(-(x)_\kappa,z)+\frac{1}{2}\pi_R^\varepsilon (|x|)\rho((x)_\kappa,z)\}}
\end{align*}
and $\L_{(\bar{X}_0^{i,N})_{1\leq i\leq N}}=\mu_0^N$,  $\L_{(\bar{X}_0^{i})_{1\leq i\leq N}}=\mu_0^{\otimes N}$.
Let $c_1,g_\eta(r)$ be defined in \eqref{gfu} and define
$$\psi(r)=\left\{
    \begin{array}{ll}
      c_1r+\int_0^r\e^{-2K_2g_\eta(s)}\d s, & \hbox{$r\in[0,2\ell_0]$;} \\
      \psi(2\ell_0)+\psi'(2\ell_0)(r-2\ell_0), & \hbox{$r\in[2\ell_0,\infty)$.}
    \end{array}
  \right.
$$
Since $$g_\eta'(r)>0,\ \ g_\eta''(r)<0,\ \ g_\eta'''(r)>0, \ \ r\in(0,2\ell_0],$$
we conclude that
$$\psi'(r)=\left\{
    \begin{array}{ll}
      \e^{-2K_2g_\eta(2\ell_0)}+\e^{-2K_2g_\eta(r)}>0, & \hbox{$r\in[0,2\ell_0]$;} \\
      \psi'(2\ell_0), & \hbox{$r\in[2\ell_0,\infty)$,}
    \end{array}
  \right.
$$
\begin{align}\label{ps2}\psi''(r)=\left\{
    \begin{array}{ll}
      -2K_2g_\eta'(r)\e^{-2K_2g_\eta(r)}\leq 0, & \hbox{$r\in[0,2\ell_0)$;} \\
      0, & \hbox{$r\in(2\ell_0,\infty)$,}
    \end{array}
  \right.
\end{align}
and
$$\psi'''(r)\geq 0,\ \  \psi^{(4)}(r)\leq 0,\ \ r\in [0,2\ell_0)\cup(2\ell_0,\infty).$$
Then by Taylor's expansion, we have
\begin{align}\label{k1e}
&\psi(r-r\wedge\kappa)+ \psi(r+r\wedge\kappa)-2\psi(r)\leq 0, \ \ r\geq 0,
\end{align}
and
\begin{align}\label{MYa}
&\psi(r-r\wedge\kappa)+ \psi(r+r\wedge\kappa)-2\psi(r)\leq \psi''(r)(r\wedge\kappa)^2,\ \ r\in[0,\ell_0].
\end{align}
Since $\psi'(0)=1+c_1$ and $\psi'(r)$ is decreasing, we conclude that $\psi(r)\leq \psi'(0) r$ and $\frac{\psi(r)}{r}$ is decreasing, which yields
\begin{align}\label{cyt}\inf_{r\geq 0}\frac{\psi(r)}{r}=\lim_{r\to\infty}\frac{\psi(r)}{r}=\lim_{r\to\infty}\psi'(r)=\psi'(2\ell_0)=2 c_1,
\end{align}
and hence we have \begin{align}\label{copa}2c_1r\leq \psi(r)\leq (1+c_1)r.
\end{align}
Define $F_{\psi}(x,y)=\psi(|x-y|), x, y\in\R^d$. It is not difficult to see from \eqref{sys} that
\begin{align}\label{sys12}(\scr L_SF_{\psi})(x,y)=0.
\end{align}
Moreover, by \eqref{rbc} and \eqref{xka}, we get
\begin{align*}&(\scr L_RF_{\psi})(x,y)\\
&=\frac{1}{2}\int_{\R^d}\left\{\psi(|y-x+(x-y)_\kappa|)-\psi(|x-y|)\right\}\nu_{(y-x)_\kappa}^\alpha(\d z)\\
&+\frac{1}{2}\int_{\R^d}\left\{\psi(|y-x+(y-x)_\kappa|)-\psi(|x-y|)\right\}\nu_{(x-y)_\kappa}^\alpha(\d z)\\
&=\frac{1}{2}\nu_{(y-x)_\kappa}^\alpha(\R^d)\left\{\psi(|y-x|-|y-x|\wedge\kappa)+ \psi(|y-x|+|y-x|\wedge\kappa)-2\psi(|x-y|)\right\}.
\end{align*}
This together with \eqref{k1e} and the fact
$J^\alpha(|x-y|\wedge \kappa)\leq \nu_{(y-x)_\kappa}^\alpha(\R^d)$ due to \eqref{Jst} and \eqref{nua}
implies that
\begin{align}\label{LRt}\nonumber&(\scr L_RF_{\psi})(x,y)\\
&\leq \frac{1}{2}J^\alpha(|x-y|\wedge \kappa)\\
\nonumber&\qquad\quad\times \left\{\psi(|y-x|-|y-x|\wedge\kappa)+ \psi(|y-x|+|y-x|\wedge\kappa)-2\psi(|x-y|)\right\}.
\end{align}
In addition, for $r\in[0,\ell_0]$, we derive from \eqref{MYa}, \eqref{ps2}, \eqref{tij},
$g_\eta'(r)=\left(1+\frac{K_1}{K_2}\right)\frac{1}{\sigma_\eta(r)}$ and \eqref{copa} that
\begin{align}\label{cth}\nonumber&\psi'(r)K_1r+\frac{1}{2}J^\alpha(r\wedge \kappa)\left\{\psi(r-r\wedge\kappa)+ \psi(r+r\wedge\kappa)-2\psi(r)\right\}\\
\nonumber&\leq \psi'(r)K_1r+\frac{1}{2}J^\alpha(r\wedge \kappa)\psi''(r)(r\wedge\kappa)^2\\
&\leq \{\e^{-2K_2g_\eta(2\ell_0)}+\e^{-2K_2g_\eta(r)}\}K_1 r+\psi''(r)r\sigma_\eta(r)\\
\nonumber&\leq \{\e^{-2K_2g_\eta(2\ell_0)}+\e^{-2K_2g_\eta(r)}\}K_1 r-(2K_2+2K_1)r\e^{-2K_2g_\eta(r)}\\
\nonumber&\leq -2K_2r\e^{-2K_2g_\eta(r)}\leq -2c_1K_2 r\leq -\frac{2c_1K_2}{1+c_1} \psi(r).
\end{align}
For $r\in[\ell_0,\infty)$, in view of \eqref{k1e}, \eqref{copa} and $\psi'(r)\geq 2c_1$ due to \eqref{cyt}, we have
\begin{align}\label{ctha}\nonumber&-\psi'(r)K_2r+\frac{1}{2}J^\alpha(r\wedge \kappa)\left\{\psi(r-r\wedge\kappa)+ \psi(r+r\wedge\kappa)-2\psi(r)\right\}\\
&\leq -2c_1K_2r\leq -\frac{2c_1K_2}{1+c_1} \psi(r).
\end{align}
Combining \eqref{LRt}, \eqref{cth} and \eqref{ctha}, we conclude that
\begin{align}\label{mcy}\nonumber&\psi'(|x-y|)(K_1|x-y|1_{\{|x-y|\leq \ell_0\}}-K_2|x-y|1_{\{|x-y|> \ell_0\}})+(\scr L_RF_\psi)(x,y)\\
&\leq -\frac{2c_1K_2}{1+c_1}\psi(|x-y|).
\end{align}
Next, by \eqref{PPY01}-\eqref{PPT02}, It\^{o}'s formula and noting that $(\scr L^{\varepsilon}F_{\psi})(x,y)=\pi_R^\varepsilon(|x-y|) (\scr L_R F_\psi)(x,y)$ due to \eqref{sys12}, we have
\begin{align}\label{pzn}
\nonumber\d \psi(|\bar{Z}_t^{i,N}|)&\leq \psi'(|\bar{Z}_t^{i,N}|)\left\<b^{0}(\bar{X}_t^{i})-b^{0}(\bar{X}_t^{i,N}), \frac{\bar{Z}_t^{i,N}}{|\bar{Z}_t^{i,N}|}1_{\{|\bar{Z}_t^{i,N}|\neq 0\}}\right\>\d t\\
&+K_b\psi'(|\bar{Z}_t^{i,N}|)|\bar{Z}_t^{i,N}|\d t+\psi'(|\bar{Z}_t^{i,N}|)\frac{1}{N}\sum_{m=1}^N K_b|\bar{Z}_t^{m,N}|\d t\\
\nonumber&+\psi'(|\bar{Z}_t^{i,N}|)\left|\frac{1}{N}\sum_{m=1}^N b^{(1)}(\bar{X}_t^{i},\bar{X}_t^{m})-\int_{\R^d}b^{(1)}(\bar{X}_t^{i},y)\mu_t(\d y)\right|\d t\\
\nonumber&+\pi_R^\varepsilon(|\bar{Z}_t^{i,N}|) (\scr L_R F_\psi)(\bar{X}_t^{i,N},\bar{X}_t^i)\d t+\d M_t^i
\end{align}
for some martingale $M_t^i$.
Observing that $\psi$ is increasing, we derive from \eqref{mcy} that
\begin{align*}&\psi'(|x-y|)(K_1|x-y|1_{\{|x-y|\leq \ell_0\}}-K_2|x-y|1_{\{|x-y|> \ell_0\}})+\pi_R^\varepsilon(|x-y|)(\scr L_RF_\psi)(x,y)\\
&\leq -\frac{2c_1K_2}{1+c_1}\psi(|x-y|)+\frac{2c_1K_2}{1+c_1}\psi(\varepsilon)+K_1\psi'(0)(\varepsilon\wedge \ell_0).
\end{align*}
This together with \eqref{pzn}, \eqref{pdic}, \eqref{copa} and $\psi'(r)\leq \psi'(0)=1+c_1$ gives
\begin{align*}
\d \sum_{i=1}^N\psi(|\bar{Z}_t^{i,N}|)&\leq -\lambda\sum_{i=1}^N\psi(|\bar{Z}_t^{i,N}|)\d t+\sum_{i=1}^N\d M_t^i+N\frac{2c_1K_2}{1+c_1}\psi(\varepsilon)+NK_1\psi'(0)(\varepsilon\wedge \ell_0)\\
&+\sum_{i=1}^N(1+c_1)\left|\frac{1}{N}\sum_{m=1}^N b^{(1)}(\bar{X}_t^{i},\bar{X}_t^{m})-\int_{\R^d}b^{(1)}(\bar{X}_t^{i},y)\mu_t(\d y)\right|\d t
\end{align*}
for $\lambda=\frac{2c_1K_2}{1+c_1}-2K_b\frac{1+c_1}{2c_1}$. Therefore, it holds
\begin{align}\label{tya}\nonumber\sum_{i=1}^N\E \psi(|\bar{Z}_t^{i,N}|)&\leq \e^{-\lambda t}\sum_{i=1}^N\E \psi(|\bar{Z}_0^{i,N}|)\\
&+\lambda^{-1}N\left\{\frac{2c_1K_2}{1+c_1}\psi(\varepsilon)+K_1\psi'(0)(\varepsilon\wedge \ell_0)\right\}\\
\nonumber&+\int_0^t\e^{-\lambda(t-s)}N(1+c_1)\left|\frac{1}{N}\sum_{m=1}^N b^{(1)}(\bar{X}_s^{i},\bar{X}_s^{m})-\int_{\R^d}b^{(1)}(\bar{X}_s^{i},y)\mu_s(\d y)\right|\d s.
\end{align}
Finally, by the similar argument in the proof of \cite[Proposition 1.5]{CDSX}, we get
\begin{align}\label{Ctq}\sup_{t\geq 0}\E|\bar{X}_t^i|^{1+\delta}<c_0(1+\E|\bar{X}_0^i|^{1+\delta})
\end{align}
for some constant $c_0>0$.
In fact, by Young's inequality
\begin{align}\label{yit}
a^{\frac{\delta}{1+\delta}}c^{\frac{1}{1+\delta}}\leq \frac{\delta}{1+\delta}a+\frac{1}{1+\delta}c,\ \ a,c\geq 0,
\end{align}
we derive
\begin{align}\label{yiq}
(1+|\bar{X}_t^i|^2)^{\frac{\delta}{2}}\E|\bar{X}_t^i|\leq \frac{\delta}{1+\delta} (1+|\bar{X}_t^i|^2)^{\frac{1+\delta}{2}}+\frac{1}{1+\delta}\E(1+|\bar{X}_t^i|^2)^{\frac{1+\delta}{2}}.
\end{align}
Then applying It\^{o}'s formula, {\bf (B2)}, \cite[(A.2)]{CDSX}, \eqref{yiq} and  $K_b<\frac{K_2}{2}$ due to $K_b<\frac{2c_1^2K_2}{(1+c_1)^2}$ with  $c_1\in(0,1]$,
we can find constants $\tilde{c}_0, \tilde{c}_1>0$ such that
\begin{align*}
&\d (1+|\bar{X}_t^i|^2)^{\frac{1+\delta}{2}}\\
&\leq (1+\delta)(1+|\bar{X}_t^i|^2)^{\frac{\delta-1}{2}}\<b^{0}(\bar{X}_t^i)-b^{(0)}(0), \bar{X}_t^i\>\d t+(1+\delta)(1+|\bar{X}_t^i|^2)^{\frac{\delta-1}{2}}\<b^{(0)}(0), \bar{X}_t^i\>\d t\\
&+(1+\delta)(1+|\bar{X}_t^i|^2)^{\frac{\delta}{2}}\left|\int_{\R^d}b^{(1)}(\bar{X}_t^i,y)\mu_t(\d y)\right|\d t-(-\Delta)^{\frac{\alpha}{2}}(1+|\cdot|^2)^{\frac{1+\delta}{2}}(\bar{X}_t^i)\d t+\d M_t\\
&\leq (1+\delta)(K_1+K_2)\ell_0^{1+\delta}\d t-(1+\delta)K_2(1+|\bar{X}_t^i|^2)^{\frac{\delta+1}{2}}\d t+(1+\delta)K_2\d t\\
&+(1+\delta)|b^{(0)}(0)|(1+|\bar{X}_t^i|^2)^{\frac{\delta}{2}}\d t+(1+\delta)(1+|\bar{X}_t^i|^2)^{\frac{\delta}{2}}|b^{(1)}(0,0)|\d t\\
&+(1+\delta)K_b(1+|\bar{X}_t^i|^2)^{\frac{1+\delta}{2}}\d t +(1+\delta)K_b(1+|\bar{X}_t^i|^2)^{\frac{\delta}{2}}\E|\bar{X}_t^i|\d t\\
&+\tilde{c}_0\d t+\tilde{c}_0(1+|\bar{X}_t^i|^2)^{\frac{\delta}{2}}\d t+\d M_t\\
&\leq \{(1+\delta)(K_1+K_2)\ell_0^{1+\delta}+(1+\delta)K_2+\tilde{c}_0\}\d t\\
&+\{(1+\delta)(|b^{(0)}(0)|+|b^{(1)}(0,0)|)+\tilde{c}_0\}(1+|\bar{X}_t^i|^2)^{\frac{\delta}{2}}\d t\\
&-(1+\delta)\left(K_2-2K_b\right)(1+|\bar{X}_t^i|^2)^{\frac{1+\delta}{2}}\d t\\
&+K_b\E(1+|\bar{X}_t^i|^2)^{\frac{1+\delta}{2}}\d t-K_b(1+|\bar{X}_t^i|^2)^{\frac{1+\delta}{2}}\d t+\d M_t\\
&\leq \tilde{c}_1\d t-\frac{(1+\delta)\left(K_2-2K_b\right)}{2}(1+|\bar{X}_t^i|^2)^{\frac{1+\delta}{2}}\d t\\
&+K_b\E(1+|\bar{X}_t^i|^2)^{\frac{1+\delta}{2}}\d t-K_b(1+|\bar{X}_t^i|^2)^{\frac{1+\delta}{2}}\d t+\d M_t
\end{align*}
for some martingale $M_t$, where in the last step, we again use \eqref{yit} to derive
\begin{align*}&\{(1+\delta)(|b^{(0)}(0)|+|b^{(1)}(0,0)|)+\tilde{c}_0\}(1+|\bar{X}_t^i|^2)^{\frac{\delta}{2}}\\
&= \left(\{(1+\delta)(|b^{(0)}(0)|+|b^{(1)}(0,0)|)+\tilde{c}_0\}^{1+\delta}\left(\frac{(K_2-2K_b)(1+\delta)^2}{2\delta} \right)^{-\delta}\right)^{\frac{1}{1+\delta}}\\
&\qquad\quad\times\left(\frac{(K_2-2K_b)(1+\delta)^2}{2\delta} (1+|\bar{X}_t^i|^2)^{\frac{1+\delta}{2}}\right)^{\frac{\delta}{1+\delta}}\\
&\leq \frac{1}{1+\delta} \left(\{(1+\delta)(|b^{(0)}(0)|+|b^{(1)}(0,0)|)+\tilde{c}_0\}^{1+\delta}\left(\frac{(K_2-2K_b)(1+\delta)^2}{2\delta} \right)^{-\delta}\right)\\
&+\frac{(K_2-2K_b)(1+\delta)}{2} (1+|\bar{X}_t^i|^2)^{\frac{1+\delta}{2}}.
\end{align*}
Therefore, \eqref{Ctq} holds.
Letting $\varepsilon\to 0$ in \eqref{tya} and using \eqref{copa}, \eqref{Ctq} and Lemma \ref{CTY} below, we arrive at
\begin{align*}
\widetilde{\W}_1((P_t^{[N],N})^\ast\mu^N_0,(P_t^\ast\mu_0)^{\otimes N})
&\leq  c\e^{-\lambda t}\sum_{i=1}^N\E|\bar{Z}_0^{i,N}|+cN\{\mu_0(|\cdot|^{1+\delta})\}^{\frac{1}{1+\delta}}N^{-\frac{\delta}{1+\delta}}.
\end{align*}
Therefore, the proof is completed by taking infimum with respect to $(\bar{X}_0^{i,N},\bar{X}_0^i)_{1\leq i\leq N}$ satisfying $\L_{(\bar{X}_0^{i,N})_{1\leq i\leq N}}=\mu_0^N, \L_{(\bar{X}_0^{i})_{1\leq i\leq N}}=\mu_0^{\otimes N}$.

\end{proof}
\begin{appendix}
\section{Appendix }
In this section, we give some auxiliary lemmas used in the proof of the main results. The first lemma involves in the quantitative convergence rate of the law of large number in $L^1(\P)$ for i.i.d. real-valued random variables which may have infinite second moment. Note that Lemma \ref{TMY} below for $\varepsilon=1$ is known and is easy to be verified.
\begin{lem} \label{TMY} Assume that $(\xi_i)_{i\geq 1}$ are i.i.d. real-valued random variables with $\E|\xi_1|^{1+\varepsilon}<\infty$ for some $\varepsilon\in(0,1)$. Then there exists a constant $C>0$ only depending on $\varepsilon$ such that
\begin{align*}
\E\left|\frac{1}{N}\sum_{i=1}^N\xi_i-\E\xi_1\right|\leq C\{\E|\xi_1|^{1+\varepsilon}\}^{\frac{1}{1+\varepsilon}}N^{-\frac{\varepsilon}{1+\varepsilon}}.
\end{align*}
\end{lem}
\begin{proof} For any $a>0$, define $\xi_i^{(a)}=(\xi_i\wedge a)\vee(-a), i\geq 1$. Since $(\xi_i)_{i\geq 1}$ are i.i.d., we conclude that
\begin{align*}
&\E\left|\frac{1}{N}\sum_{i=1}^N\xi_i-\E\xi_1\right|\\
&\leq \E\left|\frac{1}{N}\sum_{i=1}^N\xi_i-\frac{1}{N}\sum_{i=1}^N\xi_i^{(a)}\right| +\E\left|\frac{1}{N}\sum_{i=1}^N\xi_i^{(a)}-\E\xi_1^{(a)}\right|+\left|\E\xi_1^{(a)}-\E\xi_1\right|\\
&\leq \E\left|\frac{1}{N}\sum_{i=1}^N\xi_i^{(a)}-\E\xi_1^{(a)}\right|+2\E|\xi_1^{(a)}-\xi_1|=:I_1+I_2.
\end{align*}
It follows from Markov's inequality that
$$\E|\xi_i^{(a)}|^2=\E(|\xi_i|^21_{|\xi_i|<a})+a^2\P(|\xi_i|\geq a)\leq a^{1-\varepsilon}\E|\xi_i|^{1+\varepsilon}+a^{1-\varepsilon}\E|\xi_i|^{1+\varepsilon} =2a^{1-\varepsilon}\E|\xi_i|^{1+\varepsilon}, \ \ i\geq1. $$
Since $\E|\xi_1^{(a)}|^2<\infty$, we have
\begin{align*}
I_1=\E\left|\frac{1}{N}\sum_{i=1}^N\xi_i^{(a)}-\E\xi_1^{(a)}\right|
&\leq \left(\E\left|\frac{1}{N}\sum_{i=1}^N\xi_i^{(a)}-\E\xi_1^{(a)}\right|^2\right)^{\frac{1}{2}}\\
&=N^{-\frac{1}{2}}\sqrt{var(\xi_1^{(a)})}
\leq \sqrt{\E|\xi_1^{(a)}|^2}N^{-\frac{1}{2}}\leq \sqrt{2a^{1-\varepsilon}\E|\xi_1|^{1+\varepsilon}}N^{-\frac{1}{2}}.
\end{align*}
Next, we deal with $I_2$. By the definition of $\xi_1^{(a)}$, H\"{o}lder's inequality and Markov's inequality, we derive
\begin{align*}
\E|\xi_1^{(a)}-\xi_1|\leq \E[|\xi_1|1_{|\xi_1|\geq a}]&\leq [\E|\xi_1|^{1+\varepsilon}]^{\frac{1}{1+\varepsilon}}\P(|\xi_1|>a)^{\frac{\varepsilon}{1+\varepsilon}}\leq \E|\xi_1|^{1+\varepsilon}a^{-\varepsilon}.
\end{align*}
Therefore, we have
\begin{align}\label{mys}
\E\left|\frac{1}{N}\sum_{i=1}^N\xi_i-\E\xi_1\right|
&\leq\inf_{a> 0}\left\{ \sqrt{2a^{1-\varepsilon}\E|\xi_1|^{1+\varepsilon}}N^{-\frac{1}{2}}+ 2\E|\xi_1|^{1+\varepsilon}a^{-\varepsilon}\right\}.
\end{align}
Define the function $G:(0,\infty)\to\R$ as $$G(a)=\sqrt{2a^{1-\varepsilon}\E|\xi_1|^{1+\varepsilon}}N^{-\frac{1}{2}}+ 2\E|\xi_1|^{1+\varepsilon}a^{-\varepsilon},\ \ a>0.$$
It is easy to see that
$$G'(a)=\sqrt{2\E|\xi_1|^{1+\varepsilon}}N^{-\frac{1}{2}}\frac{1-\varepsilon}{2}a^{\frac{1-\varepsilon}{2}-1} -\varepsilon2\E|\xi_1|^{1+\varepsilon}a^{-\varepsilon-1}=0$$
implies
$$a=\left(\frac{2\sqrt{2}\varepsilon} {1-\varepsilon}\right)^{\frac{2}{1+\varepsilon}}\{\E|\xi_1|^{1+\varepsilon}\}^{\frac{1}{1+\varepsilon}} N^{\frac{1}{1+\varepsilon}}.$$
Substituting this into \eqref{mys}, there exits a constant $C$ only depending on $\varepsilon$ such that
\begin{align*}
\E\left|\frac{1}{N}\sum_{i=1}^N\xi_i-\E\xi_1\right|
&\leq C\{\E|\xi_1|^{1+\varepsilon}\}^{\frac{1}{1+\varepsilon}}N^{-\frac{\varepsilon}{1+\varepsilon}}.
\end{align*}
Therefore, the proof is completed.
\end{proof}
Basing on Lemma \ref{TMY}, we now show a crucial lemma to the proof of propagation of chaos.
\begin{lem}\label{CTY} Let $(V,\|\cdot\|_V)$ be a Banach space. $(Z_i)_{i\geq 1}$ are i.i.d. $V$-valued random variables with $\E\|Z_1\|_V^{1+\varepsilon}<\infty$ for some $\varepsilon\in(0,1)$ and $h:V\times V\to \R$ is  measurable and of at most linear growth, i.e. there exists a constant $K_h>0$ such that
\begin{align}\label{lig}|h(v,\tilde{v})|\leq K_h(1+\|v\|_V+\|\tilde{v}\|_V),\ \ v,\tilde{v}\in V.
\end{align}
Then there exists a constant $\tilde{c}>0$ only depending on $\varepsilon$ and $K_h$ such that
\begin{align*}\E\left|\frac{1}{N}\sum_{m=1}^N h(Z_1,Z_m)-\int_{V} h(Z_1,y)\L_{Z_1}(\d y)\right|\leq \tilde{c}\{1+\{\E\|Z_1\|_{V}^{1+\varepsilon}\}^{\frac{1}{1+\varepsilon}}\}N^{-\frac{\varepsilon}{1+\varepsilon}}.
\end{align*}
\end{lem}
\begin{proof} Applying Lemma \ref{TMY} and Jensen's inequality, we can find a constant $c_1>0$ only depending on $\varepsilon$ such that
$$\E\left\{\left\{\E\left|\frac{1}{N}\sum_{m=1}^N h(x,Z_m)-\int_{V} h(x,y)\L_{Z_1}(\d y)\right|\right\}\Bigg|_{x=Z_1}\right\}\leq c_1\{\E|h(Z_1,Z_2)|^{1+\varepsilon}\}^{\frac{1}{1+\varepsilon}}N^{-\frac{\varepsilon}{1+\varepsilon}}.$$
This together with \eqref{lig} and the fact that $(Z_i)_{i\geq 1}$ are i.i.d. yields that
\begin{align*}&\E\left|\frac{1}{N}\sum_{m=1}^N h(Z_1,Z_m)-\int_{V}  h(Z_1,y)\L_{Z_1}(\d y)\right|\\
&\leq \E\left|\frac{1}{N} h(Z_1,Z_1)\right|+\E\left|\frac{1}{N}\sum_{m=2}^N h(Z_1,Z_m)-\int_{V} h(Z_1,y)\L_{Z_1}(\d y)\right|\\
&= \E\left|\frac{1}{N} h(Z_1,Z_1) \right|\\
&\qquad\quad+\E\left\{\left\{\E\left|\frac{1}{N}\sum_{m=2}^N h(x,Z_m)-\int_{V} h(x,y)\L_{Z_1}(\d y)\right|\right\}\Bigg|_{x=Z_1}\right\}\\
&\leq  \E\left|\frac{1}{N} h(Z_1,Z_1) \right|\\
&\qquad\quad+\E\left\{\left\{\E\left|\frac{1}{N}\sum_{m=1}^N h(x,Z_m)-\int_{V} h(x,y)\L_{Z_1}(\d y)\right|+\E\left|\frac{1}{N}h(x,Z_1)\right|\right\}\Bigg|_{x=Z_1}\right\}\\
&\leq  \frac{1}{N}\E\left| h(Z_1,Z_1) \right|+\frac{1}{N}\E|h(Z_2,Z_1)|\\
&\qquad\quad+\E\left\{\left\{\E\left|\frac{1}{N}\sum_{m=1}^N h(x,Z_m)-\int_{V} h(x,y)\L_{Z_1}(\d y)\right|\right\}\Bigg|_{x=Z_1}\right\}\\
&\leq K_h\frac{1}{N}(2+4\E\|Z_1\|_V)+ c_1\{\E|h(Z_2,Z_1)|^{1+\varepsilon}\}^{\frac{1}{1+\varepsilon}}N^{-\frac{\varepsilon}{1+\varepsilon}}\\
&\leq \tilde{c}\{1+\{\E\|Z_1\|_{V}^{1+\varepsilon}\}^{\frac{1}{1+\varepsilon}}\}N^{-\frac{\varepsilon}{1+\varepsilon}}
\end{align*}
for some constant $\tilde{c}>0$ only depending on $\varepsilon$ and $K_h$. So, we complete the proof.
\end{proof}
By modifying \cite[part (c) of proof of Theorem 2.1]{WW} for the additive noise case, we provide a lemma for the Yosida approximation in the multiplicative Brownian motion noise below. Since we have not found any references on it, we give the proof for completeness.

Let $W_t$ be an $n$-dimensional Brownian motion on some complete filtration probability space $(\Omega, \scr F, (\scr F_t)_{t\geq 0},\P)$. $b:\R^d\to\R^d$ and $\sigma:\R^d\to\R^d\otimes\R^n$ are measurable.
Consider
 \begin{align}\label{ghw}\d X_t=b(X_t)\d t+\sigma(X_t)\d W_t.
\end{align}
We make the following assumptions on $b$ and $\sigma$.
\begin{enumerate}
\item[\bf{(H)}] There exists a constant $C_\sigma>0$ such that
\begin{align}\label{bslipsg1}
\|\sigma(x_1)-\sigma(x_2)\|^2_{HS}\leq C_\sigma|x_1-x_2|^2, \ \ x_1,x_2\in\R^d.
\end{align}
$b$ is continuous and there exists $C_b\in\R$ such that
\begin{align*}
&2\langle x_1-x_2, b(x_1)-b(x_2)\rangle\leq C_b|x_1-x_2|^2,\ \ x_1,x_2\in\R^d.
\end{align*}
\end{enumerate}
Let
    $$
        \tilde{b}(x):=b(x)
        -\frac12 C_bx,\ \ x\in\R^{d}.
    $$
    Then $\tilde{b}$ is also continuous.
    For any $n\geq 1$, let
    $$
        \tilde{b}^{(n)}(x)
        :=n\left[
        \left(\operatorname{id}
        -\frac{1}{n}\tilde{b}\right)^{-1}(x)-x
        \right],\quad x\in\R^d,
    $$
    where $\mathrm{id}$ is the identity map on $\R^d$.
    Then for any $n\geq 1$, $\tilde{b}^{(n)}$ is globally Lipschitz continous,
    $|\tilde{b}^{(n)}|\leq
    |\tilde{b}|$ and
    $\lim_{n\to\infty}
    \tilde{b}^{(n)}=
    \tilde{b}$. Let
    $b^{(n)}(x):=
    \tilde{b}^{(n)}(x)
    +\frac12 C_bx, x\in\R^d$. Then $b^{(n)}$ satisfies
    \begin{equation}\label{jg4dv7}
        2\<b^{(n)}(x)-
        b^{(n)}(y),
        x-y\>\leq C_b|x-y|^2,\ \ n\geq 1, x,y\in\R^d.
    \end{equation}
     Let $X_t^{(n)}$ solve \eqref{ghw} with $b^{(n)}$ replacing $b$ and $X_0^{(n)}=X_0$.
\begin{lem}\label{yap} Assume {\bf(H)}. Then for any $t\geq 0$, there exists a subsequence $\{n_k\}_{k\geq 1}$ such that $\P$-a.s. $X_t^{(n_k)}$ converges to $X_t$ as $k\to\infty$.
\end{lem}
\begin{proof} Fix $t\geq 0$.
It is standard to derive from \eqref{jg4dv7} and \eqref{bslipsg1} that
$$\sup_{n\geq 1}\E|X_t^{(n)}|^2<\infty.$$
    By It\^{o}'s formula, we derive
    \begin{align*}
        &\d|X_t^{(n)}-X_t|^2\\
        &=2\<X_t^{(n)}-X_t,
        b^{(n)}(
        X_t^{(n)})
        -b^{(n)}(X_t)\>\,\d t+2\<X_t^{(n)}-X_t,
        b^{(n)}(
        X_t)
        -b(X_t)\>\,\d t\\
        &\quad+2\<X_t^{(n)}-X_t,
        \{\sigma(
        X_t^{(n)})
        -\sigma(X_t)\}\d W_t\>+\|\sigma(
        X_t^{(n)})
        -\sigma(X_t)\|^2_{HS}\d t\\
        &\leq (C_b+1+C_\sigma)|X_t^{(n)}-X_t|^2\d t
        +\big|b^{(n)}(
        X_t)
        -b(X_t)
        \big|^2\,\d t\\
        &+2\<X_t^{(n)}-X_t,
        \{\sigma(
        X_t^{(n)})
        -\sigma(X_t)\}\d W_t\>.
    \end{align*}
    For any $m\geq 1$, define $\tau_m=\inf\{t\geq 0, |X_t|\geq m\}$. Then Gronwall's inequality implies that
\begin{align*}
       \E |X_{t\wedge\tau_m}^{(n)}-X_{t\wedge\tau_m}|^2
        &\leq
        \e^{(C_b+1+C_\sigma)t}\E\int_0^{t\wedge \tau_m}\big|\tilde{b}^{(n)}(
        X_s)
        -\tilde{b}(X_s)
        \big|^2\,\d s
\end{align*}
    Since $\tilde{b}$ is continuous and $|\tilde{b}^{(n)}|\leq |\tilde{b}|$ and $\lim_{n\to\infty}\tilde{b}^{(n)}=\tilde{b}$, we may use the
    dominated convergence theorem to derive
    that for any $m\geq 1$,
    $$\lim_{n\to\infty}\E |X_{t\wedge\tau_m}^{(n)}-X_{t\wedge\tau_m}|^2=0.$$
    This implies that for any $m\geq 1$, there exists a subsequence $\{n^{(m)}_k\}_{k\geq 1}$ and  $\Omega_m\in\F$ with $\P(\Omega_m)=1$ such that $$\lim_{k\to\infty}
    X_{t\wedge\tau_m}^{(n^{(m)}_k)}(\omega)=X_{t\wedge\tau_m}(\omega),\ \ \omega\in\Omega_m. $$
    By the diagonal argument, we can find a subsequence $\{n_k\}_{k\geq 1}$ and $\bar{\Omega}\in\scr F$ with $\P(\bar{\Omega})=1$ such that for any $m\geq 1$, $$\lim_{k\to\infty}
    X_{t\wedge\tau_m}^{(n_k)}(\omega)=X_{t\wedge\tau_m}(\omega),\ \ \omega\in\bar{\Omega}. $$
    In view of $\P$-a.s. $\tau_m\uparrow\infty$ as $m\uparrow\infty$, we conclude that $\P$-a.s. $\lim_{k\to\infty}
    X_{t}^{(n_k)}=X_{t}$. The proof
    is now finished.
    \end{proof}
\end{appendix}

\end{document}